\documentclass[final]{article}

\usepackage{latexsym, amsmath, amsfonts,  amssymb, mathtools}
\usepackage{subfigure}
\usepackage{epsfig}
\usepackage{color}
\usepackage{showkeys}
\usepackage{comment}
\usepackage{graphicx}

\newcommand{\bn}{\ensuremath{\mathbf{n}}}
\newcommand{\bm}{\ensuremath{\mathbf{m}}}

\newcommand{\bx}{\ensuremath{\mathbf{x}}}

\newcommand{\bD}{\ensuremath{\mathbf{D}}}

\newcommand{\bh}{\ensuremath{\mathbf{h}}}
\newcommand{\vp}{\ensuremath{\varphi}}
\newcommand{\bR}{\ensuremath{\mathbb{R}}}
\newcommand{\pa}{\ensuremath{\partial}}
\newcommand{\e}{\ensuremath{\varepsilon}}

\newcommand{\mc}{\ensuremath{\mathcal}}
\newcommand{\omm}{\ensuremath{\Omega}}
\newcommand{\nb}{\ensuremath{\nabla}}
\newcommand{\fr}{\ensuremath{\frac}}

\newcommand{\nt}{\ensuremath{\negthickspace}}
\newcommand{\Dn}{\ensuremath{\bD_{\bn}}}

\newcommand{\TT}{\ensuremath{\mathbb{T}^2}}

\newcommand{\ttheta}{\ensuremath{\tilde{\theta}}}
\newcommand{\htheta}{\ensuremath{\hat{\theta}}}

\DeclareMathOperator{\arsinh}{arsinh}

\newtheorem{theorem}{Theorem}[section]
\newtheorem{lemma}{Lemma}[section]

\newtheorem{remark}{Remark}[section]
\newtheorem{proof}{Proof}[section]
\newtheorem{proposition}{Proposition}[section]

\begin{document}

\pagestyle{plain}\title{\textbf{Sawtooth profile in Smectic A liquid crystals}}


\author{Carlos J. Garc{\'\i}a-Cervera\thanks
         {Mathematics Department,
         University of California,
         Santa Barbara, CA 93106
         (cgarcia@math.ucsb.edu).
         This author's research was supported by NSF grants DMS-0645766.}
        \and
        Tiziana Giorgi\thanks
          {Department of Mathematical Sciences,
           New Mexico State University,
           Las Cruces, NM 88001
           (tgiorgi@nmsu.edu).
          Funding to this author was provided by the National Science
Foundation Grant \#DMS-1108992 }
        \and
        Sookyung Joo\thanks
         {Department of Mathematics and Statistics,
          Old Dominion University
         Norfolk, VA 23529
         (sjoo@odu.edu).
         This author was supported by the NSF-AWM Mentoring Travel Grant and ODU SRFP grant.}}

\maketitle

\begin{abstract}
{\sl
We study the de Gennes free energy for smectic A liquid crystals over $\mathbb{S}^2$-valued vector fields to understand the chevron (zigzag) pattern formed  in the presence of an applied magnetic field. We identify a small dimensionless parameter $\e$, and investigate the behaviors of the minimizers  when the field strength is of order $\mathcal{O}(\e^{-1})$. In this regime, we show  via $\Gamma$-convergence that a chevron structure where the director connects two minimum states of the sphere is favored. We also analyze the Chen-Lubensky free energy, which includes the second order gradient of the smectic order parameter, and obtain the same general behavior as for the de Gennes case. Numerical simulations illustrating the chevron structures for both energies are included.
  }
\end{abstract}

\pagestyle{myheadings}
\thispagestyle{plain}
\markboth{C. J. Garc{\'\i}a-Cervera AND T. Giorgi AND S. Joo}{ Sawtooth profile in smectic A liquid crystals}

\section{Introduction}
The rod-like molecules of a liquid crystal in the smectic~A phase tend to align with each other, and arrange themselves into equally spaced layers, perpendicular to the principal molecular axis. If the liquid crystal sample is confined between two flat plates, and its molecules are uniformly aligned so that the smectic layers are parallel to the bounding plates, a magnetic field applied in the direction parallel to the layer will tend to reorient the molecules and the layers, while the surface anchoring condition at the plates will opposite this reorientation. Hence, an instability will occur above a threshold magnetic field, called Helfrich-Hurault effect \cite{He, Hu}, where layer undulation will appear. As the applied field increases well above this first critical value, the sinusoidal shape of the smectic layer will change into a chevron (zigzag) pattern with a longer period.
The Helfrich-Hurault effect has been analyzed by Garc{\'\i}a-Cervera and Joo in \cite{G-J1, G-J2, G-J3}, where the periodic oscillations of the smectic layers and molecular alignments were described at the onset of the undulation, and the critical magnetic field was estimated in terms of the material constants and sample thickness. In this paper, we are interested in the higher field regimes,  in particular in the description and derivation of the chevron profile.

Experimental studies of the development  of the chevron pattern from the sinusoidal shape of undulation were presented by Lavrentovich {\it et al.} in \cite{I-L, S-S-L}. They also proposed a model with weak anchoring conditions for sinusoidal and sawtooth undulation profiles. By equating the director and layer normal, making an ansatz of periodic undulations, and assuming the square lattice for undulations, they reduced the problem to one dimension, obtained an ordinary differential equation, and found the explicit solution to the equation. To rigorously study the zigzag pattern in full generality, we analyze via $\Gamma$-convergence a two-dimensional de Gennes energy functional, without identifying the director with the layer normal.

The de Gennes free energy density includes nematic, smectic~A and magnetostatic contributions. A nondimensionalization procedure leads to the identification of a small parameter $\e$. In \cite{G-J2}, the authors show that the critical field of the undulation phenomenom is of order $\mathcal{O}(1)$. More precisely, they obtain estimates of  $\pi$ and $1$, with and without the assumption that the layers are fixed at the bounding plates, respectively.  In this work, we consider regimes where the field strength is of order $\mathcal{O}(\e^{-1})$.

The mathematical analysis we adopt for the two dimensional de Gennes free energy is motivated by the study of domain walls in ferromagnetism. By reformulating the free energy, we capture a double well potential having two minimum states for the director on the sphere, hence we follow \cite{Anzellotti}  and use a Modica-Mortola-type inequality on the sphere equipped with a new metric associated with the double well potential. Additionally, since experiments show periodic chevron patterns, we consider a system with periodic boundary conditions, and adapt to our problem for $\mathbb{S}^2$-valued vector fields, the variational approach on the flat torus presented in \cite{Choksi-Sternberg}, where the authors consider the Cahn-Hilliard energy in the periodic setting in order to study microphase separation of diblock copolymers.  It's important to notice that while extending the techniques in \cite{Anzellotti}  to the flat torus, we need to consider the presence of the smetic order parameter, and work with an explicit form of a geodesic curve connecting the two minima in the new metric.

We also consider the model introduced by Chen and Lubensky in 1976 \cite{C-L}, which is based on the de Gennes model for
smectic~A, but includes a second order gradient term for the smectic order parameter. In \cite{C-L}, the authors investigate the nematic to smectic~A or smectic~C phase transition, and use this model to predict the twist grain boundary phase in chiral smectic liquid crystals \cite{Renn-Lubensky}. However, it is well-known that their model lacks coercivity of the energy, therefore in here we consider its modification as presented in (\cite{Luk, J-P}). Although the $\Gamma$-convergence analysis for the two dimensional de Gennes energy in the flat torus setting can be applied to the two dimensional Chen-Lubensky energy (see Remark~\ref{Chen-Lub_remark}), we present also a $\Gamma$-convergence result for the one dimensional Chen-Lubensky energy on the interval with periodic boundary condition, not on $\mathbb{S}^1$, which we believe is of mathematical interest, as in this setting an additional boundary integral term is present in the $\Gamma$-limit.

Numerical simulations are carried out to illustrate the sawtooth profiles of undulations by solving the gradient flow equations in three dimensional space. The molecular alignment and layer structure at the cross section of the body confirm our mathematical analysis. The numerics also show that the evolution from the sinusoidal perturbation at the onset of undulations to the chevron pattern occurs with an increase of the wavelength. Numerical methods developed in \cite{G-J3}  for the Chen-Lubensky functional are employed for our problem.

Our mathematical results and numerical experiments are consistent with the experimental picture presented by Lavrentovich {\it et al.} in \cite{I-L, S-S-L}.

Chevron formation is also observed  in a surface-stabilized liquid crystal cell cooled from the smectic~A to the smectic~C phase, an interesting analytic variational characterization of this phenomenon can be found in \cite{Lei}.

The paper is organized as follows. In section \ref{section:deGennes}, we introduce the de Gennes model and the scaling regime of our problem. Then we obtain the $\Gamma$-convergence result in two dimensions for the flat torus. In section \ref{section:CL}, we prove the results for the Chen-Lubensky model. Numerical simulations for the de Gennes and Chen-Lubensky models are presented in section \ref{section:numerics}. Detailed analyses of the geodesics used in the $\Gamma$-limit analysis, and of the double well profile for the Chen-Lubensky energy are provided in appendix~\ref{geodesics}.

\section{The de Gennes model} \label{section:deGennes}
We consider the complex de Gennes free energy to study the chevron structure of smectic A liquid crystals due to the presence of a magnetic field. The smectic state is described by a unit vector $\bn$ and a complex order parameter $\psi$. The unit vector field $\bn$, or director field, represents the average direction of molecular alignment. The smectic order parameter is written as $\psi(\bx) = \rho(\bx) e^{i q \omega(\bx)}$, where $\omega$ parametrizes the layer structure so that $\nb \omega$ is perpendicular to the layer. The smectic layer density $\rho$ measures the mass density of the layers.

According to the de Gennes model, the free energy is given by
\begin{eqnarray} \label{mcG}
&&\mc{G} (\psi, \bn) =  \int_{\omm} \Big(C |\nb \psi - i q \bn \psi |^2 + K |\nb \bn|^2 + \frac{g}{2} \Big( |\psi|^2 - \frac{r}{g} \Big)^2 \nonumber \\
&&\qquad \qquad \qquad - \chi_a  H^2 (\bn \cdot \bh)^2 \Big) \, d \bx,
\end{eqnarray}
where the material parameters $C, K, g, $ and temperature dependent parameter $r = T_{NA}-T$ are fixed positive constants. The last term in (\ref{mcG}) is the magnetic free energy density,  $\bh$ is a unit vector representing the direction of the magnetic field, and $H^2$ is the strength of the applied field.  We consider a rectangular sample, $\omm= (-L, L)^2 \times (-d, d)$.

Since we are interested in the development of the layers when these are well defined, we take $ \rho = \frac{r}{g}$, with this assumption  the energy (\ref{mcG}) becomes
\begin{equation*}
\int_{\omm} \left(Cq^2 |\nb \omega - \bn |^2 + K |\nb \bn|^2  - \chi_a  H^2 (\bn \cdot \bh)^2 \right) \, d \bx.
\end{equation*}
We consider the change of variables $\bar{\bx} = \bx/d$,  and obtain the following  nondimensionalized energy
\begin{equation} \label{energy-numerics}
\mc{G} (\vp, \bn) = \frac{d K}{\e} \int_{\tilde{\omm}} \left( \frac{1}{\e} |\nb \vp -\bn |^2 + \e |\nb \bn|^2 - \tau (\bn \cdot \bh)^2 \right) \, d\bar{\bx},
\end{equation}
where
\[
\vp = \frac{\omega}{d}, \qquad \e = \frac{\lambda}{d}, \qquad \lambda =  \sqrt{\frac{K }{Cq^2 }}, \qquad \tau = \frac{\chi_a H^2 d^2 \e}{K},
\]
and
\[
\tilde{\omm} = (-l,l )^2 \times (-1,1 ), \qquad l = \frac{L}{d}.
\]
The dimensionless parameter $\e$ is in fact the ratio of the layer thickness to the sample thickness and thus $\e \ll 1$. The values $d = 1 mm$ and $\lambda = 20 \AA$ are employed in \cite{dG-P}. This small parameter $\e$ is also used in \cite{G-J2}, where the authors investigate the first instability of $\mc{G}$, and find the critical field, $\tau_c$, at which undulations appear to be $\tau_c = \mathcal{O}(1)$. In this paper we are interested in the layer and director configurations for $\tau = \mathcal{O}( \e^{-1} )$. Therefore, we set $\sigma = \tau \e$ and treat $\sigma$ as a constant.

We study the layer structure in the cross section of the sample, ($z=0$), so that the problem is reduced to a two dimensional case. Thus, we assume that
 $\bn = (n_1, n_2, n_3)$, where $\bn = \bn(x,y)$,  and the magnetic field is applied in the $x$-direction, $\bh = \mathbf{e}_1$. We rewrite the magnetic energy density as
\[
-\frac{\sigma}{\e} \bn_1^2 =  -\frac{\sigma}{\e}  + \frac{\sigma}{\e} (n_2^2 + n_3^2),
\]
and set $\vp = z-g(x,y)$ with $g$ being the layer displacement from the flat position. Up to multiplicative and additive constants, and dropping the bar notation, the energy can then be written as
\begin{equation*}
 \int_{\omm}  \left[ \e |\nb \bn|^2 +\frac{1}{ \e} W(\bn) + \frac{1}{\e} (g_x+n_1)^2+ \frac{1}{\e} (g_y+n_2)^2  \right] \, dx\ dy,
\end{equation*}
where $\omm=(-l, l)^2$, and
\begin{equation}\label{potential}
W(\bn ) =  \sigma n_2^2  +\frac{1}{A} (n_3-A)^2,
\end{equation}
with $A= (1+ \sigma)^{-1} < 1 $ .

We denote the zeros of $W : \mathbb{S}^2 \to [0, \infty)$ by
\begin{equation} \label{zero_W}
\bn^{\pm} = (\pm \bar{n}_1, 0, A),
\end{equation}
where $\bar{n}_1 = \sqrt{1-A^2}$, and  write $\alpha = \arccos(A)$.

To incorporate the periodic boundary conditions in our mathematical framework, we consider a two dimensional flat torus $\mathbb{T}^2 = \mathbb{R}^2 / (2l \mathbb{Z})^2$, that is the square $[-l,l)^2$ with periodic boundary conditions. More detailed definitions on the Sobolev space, BV spaces, and finite perimeters on $\TT$ can be found in \cite{Choksi-Sternberg}.

In conclusion, to study the chevron structure of the layer, we consider the energy functional:
\begin{equation}\label{Fe}
F_{\e} (\bn,g) = \int_{\mathbb{T}^2 }  \left[ \e |\nb \bn|^2 +\frac{1}{ \e} W(\bn) + \frac{1}{\e} (g_x+n_1)^2+ \frac{1}{\e} (g_y+n_2)^2  \right] \, dx\, dy.
\end{equation}
We analyze the configuration of the minimizers of $F_{\e}$ using $\Gamma$-convergence \cite{maso1993, Braides_book}. In particular, we use the following characterization of the $\Gamma$-limit, \cite{maso1993}:

Let $(X,\cal T)$ be a topological space, and
  ${G_h}$ be a family of functionals parameterized by $h$. A functional
  $G_0$ is the $\Gamma$-limit of ${G_h}$ as $h\to 0$ in $\cal T$ iff
  the two following conditions are satisfied:
\begin{description}
  \item [(i)] If $u_h \to u_0$ in $\cal T$, then $\liminf_{h\to 0}
    G_h(u_h) \ge G_0(u_0)$.
  \item [(ii)] For all $u_0 \in X$, there exists a sequence ${u_h}\in X$
    such that $u_h\to u_0$ in $\cal T$, and $\lim _{h\to 0} G_h(u_h) =
    G_0(u_0)$.
\end{description}
Condition $(i)$ is related to lower-semicontinuity, while to verify condition $(ii)$ a specific construction for the converging sequence is typically required.

In our setup, we will use the sets
\begin{eqnarray*}
\nt \nt \mathcal{Y} \nt &=&\nt \mathbf{W}^{1,2}(\TT, \mathbb{S}^2) \times W^{1,2}(\TT) \label{Y_3}, \\
\nt \nt \mathcal{A} \nt &=&\nt \Big \{ (\bn, g) \in \mathbf{BV}(\TT, \{\bn^{\pm} \}) \times W^{1,2}(\TT): g_x \in BV(\TT, \{\pm \bar{n}_1 \} ),  \nonumber\\
                        && \qquad    g = g(x), \,\bn = \bn(x), \, g_x+n_1 =0\;  a.e. , \, n_2=0 \; a.e., \; \int_{\TT} n_1 =0\Big \}, \label{A_3}
\end{eqnarray*}
and the following functionals, $G_{\e}$ and $G_0$, defined on $X=L^1(\TT , \mathbb{S}^2) \times L^2(\TT)$:
\begin{equation} \label{G_e}
G_{\e} (\bn, g) :=
    \begin{cases}
         F_{\e} (\bn, g)  &\mbox{if } (\bn, g) \in  \mathcal{Y}  ,\\
         +\infty &\mbox{else},
    \end{cases}
\end{equation}
and
\begin{equation} \label{G_0}
G_0 (\bn, g) :=
    \begin{cases}
          2 c_0 P_{\TT}(A_{\bn^-} ) \quad &\mbox{if } (\bn, g)   \in \mc{A} . \\
         +\infty &\mbox{else}.
    \end{cases}
\end{equation}
In the above,
 $P_{\TT}(A_{\bn^-} )$ is the perimeter, defined as in \cite{giusti1984minimal, Choksi-Sternberg}, of the  set
\begin{equation}\label{An}
A_{\bn^-} = \{(x,y) \in \TT : \bn(x) = \bn^- \},
\end{equation}
while
 \begin{equation*} \label{geodesic}
c_0 = \inf \Big \{ \int_0^1 \sqrt{W(\gamma(t))} |\gamma'(t)| dt: \gamma \in C^1([0,1], \mathbb{S}^2), \gamma(0) = \bn^-, \gamma(1) = \bn^+ \Big \}.
\end{equation*}
The value of $c_0$, as mentioned in \cite{Moser2009}, can be computed and it is given by
\begin{equation} \label{c_0}
c_0 = \frac{2}{\sqrt{A}} (\sin \alpha - \alpha \cos \alpha).
\end{equation}
For the convenience of the reader a sketch of a proof of (\ref{c_0}) is provided in appendix~\ref{geodesics} at the end of this paper.

A $\Gamma$-convergence result is always paired with a compactness property, to ensure that every cluster point of a sequence of minimizers for $G_\e$ is a minimizer of $G_0$.

\begin{proposition}  \it{(Compactness)}\label{compactness}
Let the sequences $\{\e_j\}_{j \uparrow \infty} \subset (0, \infty)$,  and $ \{ (\bn_j , g_j)\}_{j \uparrow \infty}\subset\mc{Y}$ be such that
\[
\e_j \to 0, \qquad \mbox { and } \qquad \{ G_{\e_j}(\bn_j, g_j)\}_{j \uparrow \infty} \qquad \mbox{is bounded.}
\]
Then
there exist a subsequence $\{ (\bn_{j_k}, g_{j_k}) \}$ and $(\bn, g) \in \mathcal{A}$ such that
\[
\bn_{j_k} \to \bn \mbox{ in } L^1(\TT , \mathbb{S}^2) \qquad \mbox{ and } \qquad g_{j_k} - \frac{1}{4l^2} \int_{\omm} g_{j_k} dx dy \to g \mbox{ in } L^2(\TT).
\]
\end{proposition}

\begin{proof}
The uniform bound $G_{\e_j} (\theta_j, g_j) \leq M$ and $(\bn_j, g_j) \in \mathcal{Y}$ give
\begin{equation} \label{control-n}
n_{j,2} \to 0 \qquad \mbox{ and } \qquad n_{j,3} \to A \quad \mbox{ in } \quad L^2(\TT),
\end{equation}
which lead to
$
\displaystyle{|n_{j,1}| \to \bar{n}_1  \mbox{ in } L^1(\TT),}
$
where $\bar{n}_1 = \sqrt{1-A^2}$. We define, for $\xi, \eta \in \mathbb{S}^2$:
\begin{eqnarray}\label{distance}
&& d(\xi,\eta)=\inf \bigg \{  \int_0^1 \, \sqrt{W(\gamma(t))} \,\, |\gamma'(t)| \, dt;  \mbox{ for } \gamma \in C^1([0,1]),  \nonumber \\
&&  \qquad \qquad  \qquad \gamma(t)\in \mathbb{S}^2 \mbox{ such that } \gamma(0)=\xi, \gamma(1)=\eta   \bigg \},
\end{eqnarray}
and
\begin{eqnarray}\label{Phi}
&&\Phi(\xi) = d({\bf n}_-,\xi).
\end{eqnarray}
Note that in this notation we have $c_0= d({\bf n}_-,{\bf n}_+)= \Phi({\bf n}_+)$.

Known results imply that  the function $d(\xi,\eta)$ is a distance on $\mathbb{S}^2$ associated to the degenerate Riemann metric defined by $\sqrt{W}$, and $\Phi$ is Lipschitz continuous with respect to the Euclidean distance, see \cite{Anzellotti}.

We define $w_j = \Phi(\bn_j)$, and apply the classical Modica-Mortola argument to derive
\begin{eqnarray} \label{cpt}
 \liminf_{j \to \infty} G_{\e_j} (\bn_j, g_j) &\geq& 2 \liminf_{j \to \infty} \int_{\TT} \sqrt{W(\bn_j)} |\nb \bn_j|  \nonumber \\
     &  \geq & 2 \liminf_{j \to \infty} \int_{\TT} |D (\Phi(\bn_j))| \equiv
 2 \liminf_{j \to \infty} \int_{\TT} |D w_j|,
\end{eqnarray}
where the second inequality of (\ref{cpt}) is a consequence of  Lemma~4.2 in \cite{Anzellotti}.
From this we have that  the $w_j$ are uniformly bounded in $W^{1,1}(\TT)$, and therefore there exists $w_0 \in BV(\TT)$ such that (up to subsequences) $w_j \to w_0$ in $L^1(\TT)$ and $a.e.$ in $\TT$. If we define
\[
\bn = \bn^- \chi_S + \bn^+ \chi_{\TT \setminus S},
\]
where $S = \{ \bx \in \TT : w_0(\bx) = 0 \}$,
following \cite{Modica1987,Modica-Mortola} and Proposition~4.1 in \cite{Baldo90}, we can then show that there is a subsequence $\bn_{j_k}$ such that
\[
\bn_{j_k} \to \bn \quad \mbox{in} \quad L^1(\TT , \mathbb{S}^2) \qquad \mbox{and} \qquad \bn \in BV(\TT, \bn^{\pm}).
\]
We next look at the layer displacement $g$,   the bound on the energy gives
\begin{equation} \label{control-g}
\int_{\TT} \left[ ((g_j)_x+ (n_j)_1)^2 + ((g_j)_y + (n_j)_2)^2 \right] \leq M \e,
\end{equation}
that is  $\| \nb g_j \|_2 \leq C$, so we can find a subsequence $\{g_{j_k} \}$ and a $g$ for which
\[
g_{j_k} - \frac{1}{4l^2} \int_{\omm} g_{j_k} dx dy  \rightharpoonup g \, \mbox{ in } \, H^1(\TT),
\]
and
\[
 \quad g_{j_k} - \frac{1}{4l^2} \int_{\omm} g_{j_k} dx dy  \to g \, \mbox{ in } \, L^2(\TT).
\]
Additionally, from (\ref{control-n}) and (\ref{control-g}), we have $g (x,y)= g(x)$ $a.e.$, and  since
\[
\int_{\TT} (g_x + n_1)^2 \leq \liminf_{j \to \infty} \int_{\TT} ( (g_j)_x + (n_j)_1)^2 = 0,
\]
we obtain $n_1(x,y) = -g'(x)$ $a.e.$ in $\TT$, and $\displaystyle{\int_{\TT} n_1 = - \int_{\TT} g_x = 0,}$
by the periodicity of $g$. This implies
$
\bn(x,y) = \bn(x)
$
$a.e.$, hence $(\bn, g) \in \mc{A}$.
\end{proof}

The proof of the following lower bound inequality follows directly from the proof of (step 1) of Theorem~2.4 in \cite{Anzellotti}.
\begin{lemma} \it{(Lower semi-continuity)} \label{thm:lower}
For every $(\bn, g) \in L^1(\TT , \mathbb{S}^2) \times L^2(\TT)$, and every sequence $(\bn_j, g_j) \in \mathcal{Y} $ such that $(\bn_j, g_j)$ converges to $(\bn, g) $ in $L^1(\TT)\times L^2(\TT)$, there holds
\[
       \liminf_{j \to \infty} G_{\e_j} (\bn_j, g_j) \geq G_0(\bn, g),
\]
and
$
(\bn, g)  \in \mc{A}.
$
\end{lemma}

\begin{lemma} \it{(Construction)} \label{thm:upper}
For any $(\bn, g) \in \mc{A}$, there exists a sequence $(\bn_j, g_j) \in \mathcal{Y}  $, converging in $L^1(\TT , \mathbb{S}^2) \times L^2(\TT) $ as $j \to \infty$ to $(\bn, g)$, and such that
\[
  \limsup_{j \to \infty} G_{\e_j} (\bn_j,g_j) = G_0(\bn, g).
\]
\end{lemma}
\begin{proof}
The construction of a recovering sequence combines ideas from \cite{Modica1987,Modica-Mortola}, and the proofs of (step 2) of Theorem~2.4 in \cite{Anzellotti}, and condition (2.9) in \cite{Baldo90}.

As in \cite{Anzellotti}, given $(\bn, g) \in \mc{A}$, for $(x,y)\in\TT$ we define
$$
\rho(x,y)=\Big\{ \begin{array}{ll}
 - \mbox{ dist}((x,y),\partial A_{\bn^-} ) & \mbox{ if } (x ,y)\in A_{\bn^-} , \\
 \quad \mbox{dist}((x,y),\partial A_{\bn^-} )& \mbox{ if } (x,y)\notin A_{\bn^-} .
\end{array}
$$
Note that because $\bn$ is a function only of $x$, we have $\rho(x,y_1)=\rho(x,y_2)$ for any $(x,y_1), (x,y_2)\in \TT$, that is $\rho(x,y)=\rho(x)$.

Keeping in mind Lemma~5.2 in subsection~\ref{geodesics}, we pick  $\gamma_C(t)=(\sin(2\alpha t- \alpha),0,\cos(2\alpha t-\alpha))$ to construct $\psi_j:[0,1] \to \bR$ as
\begin{equation*}
\psi_j(t)=\int_0^t \frac{2\, \alpha \epsilon_j}{\sqrt{\epsilon_j+W(\gamma_C(s))} }\, ds.
\end{equation*}
If $\eta_j=\psi_j(1)$, we have $0<\eta_j < 2 \, \epsilon_j^{1/2}\, \alpha$, and denoting by $\hat\zeta_j:[0,\eta_j]\to [0,1]$ the inverse function of $\psi_j$, we set
\begin{equation*}
\zeta_j(t)=\Bigg\{
\begin{array}{cl}
0 & \mbox{ if } t<0, \\
\hat\zeta_j(t) & \mbox{ if } 0\leq t \leq \eta_j,\\
1 &\mbox{ if } t>\eta_j. \\
\end{array}
\end{equation*}
We next consider
\begin{equation*}
\chi(t)=\Bigg\{
\begin{array}{cl}
\bn^-& \mbox{ if } t<0, \\
\bn^+ &\mbox{ if } t>0, \\
\end{array}
\end{equation*}
so that we can write $\bn(x)=\chi(\rho(x))$. This is significant, because  for every $t$ it holds $\big(\gamma_C(\zeta_j(t))\big)_1 \leq \big(\chi(t)\big)_1$ and
$\big(\chi(t)\big)_1 \leq \big(\gamma_C(\zeta_j(t+\eta_j))\big)_1$, thus there exists a $\delta_j\in[0,\eta_j]$ for which
\begin{equation*}
\int_{\TT} \big(\gamma_C(\zeta_j(\rho(x)+\delta_j))\big)_1 = \int_{\TT} \big(\chi(\rho(x))\big)_1 =\int_{\TT} n_1 =0.
\end{equation*}
We then define, for $(x,y)\in\TT$,
$\displaystyle{
{\bn}_j(x,y)=\gamma_C(\zeta_j(\rho(x)+\delta_j)),
}$
since with  this choice we have $\bn_j(x,y)=\bn_j(x)$ and $\int_{\TT} ({n}_j)_1 =0$. Therefore, by setting $g_j(x)=\int_{-l}^x (\hat{n}_j)_1 (x) \, dx $, and using the fact that  by definition of $\gamma_C$ the $y$-component of $\bn_j$ is identically equal to zero, we can  argue as in \cite{Anzellotti,Baldo90} to conclude that the sequence $(\bn_j, g_j)$ verifies the required conditions.
\end{proof}

The following theorem is a consequence of the previous $\Gamma$-convergence result, and uniqueness of the pattern of the minimizers of $G_0$. A similar interface limit for a periodic system is studied in \cite{Choksi-Sternberg}, in here we apply their method of proof.

\begin{theorem}\label{thm:deGennes}
Let $\{(\bn_{\e}, g_{\e})\} \in \mc{Y}$ be a sequence of minimizers of $G_{\e}$. Then there exists a sequence $\{c_\e\} \subset (-l, l)$ such that
\[
(\tilde{\bn}_{\e}, \tilde{g}_{\e}) \to (\bn, g) \qquad \mbox{ in } L^1(\TT , \mathbb{S}^2) \times L^2(\TT),
\]
where
\begin{eqnarray*}
&&\tilde{\bn}_{\e}(x,y) = \bn_{\e} (x+c_{\e}, y), \qquad \tilde{g}_{\e} (x,y) = g_{\e}(x+c_{\e},y ), \\
&&\bn = \bn^- \chi_{L^-} + \bn^+ \chi_{L^+},  \qquad  g_x = -\bar{n_1} \chi_{L^-} + \bar{n_1} \chi_{L^+},
\end{eqnarray*}
 $L^- = \{ x : \frac l2 < |x|  < l\} $ and $L^+ = \{ x :  |x|  < \frac l2\} $.
Furthermore, $G_0(\bn, g) = 8 c_0  r$ with $c_0$ as in (\ref{c_0}).
\end{theorem}

\begin{proof}
 Lemmas ~\ref{thm:lower}~and~\ref{thm:upper} result in
$
 G_0 = \Gamma-\lim_{\e \to 0} G_\e.
$
Also, it is clear that the minimizers of $G_0$ are vertical strips with two parallel 1d-tori, due to the condition $\int_{\TT} n_1 = 0$. In fact, these are the minimizers of the periodic isoperimetric problem studied in \cite{Choksi-Sternberg}.
 We argue by contradiction. Suppose there is $\delta >0$ and a sequence $\e_j \to 0$ such that
 \begin{equation} \label{contradiction}
 \inf_{a \in \mathbb{T}^1} \| (\bn_{\e_j}(\cdot + a,\cdot), \; g_{\e_j}  (\cdot + a,\cdot) ) - (\bn,g) \|_{L^1 (\TT, \mathbb{S}^2) \times L^2(\TT)} \geq \delta.
 \end{equation}
 Since $(\bn_{\e_j},  g_{\e_j})$ is a sequence of minimizers of $G_\e$, it follows from Proposition \ref{compactness} that there is a further subsequence $\{(\bn_{\e_j}, g_{\e_j} ) \}$, not relabeled, and a minimizer $(\bm, h) \in \mc{A}$ of $G_0$ such that $(\bn_{\e_j}, g_{\e_j})  \to (\bm, h) $ in $L^1(\TT,\mathbb{S}^2) \times L^2(\TT)$.  By the uniqueness of the pattern of the interface limit, $\bm$ has two phases separated by two vertical line segments, i.e.,
 \[
 \bm(x,y) = \bn(x+b, y)
 \]
 for some $b \in \mathbb{T}^1$, which is in contradiction with (\ref{contradiction}).

\end{proof}

\section{Chen-Lubensky energy} \label{section:CL}
In this section, to find the chevron structure for certain regimes of the magnetic field strengths we study the modified Chen-Lubensky functional for smectic A liquid crystals presented in \cite{Luk, J-P}. We start by reformulating the energy  in order to better understand how the layer evolves from the undulations to the chevron profiles.

The Chen-Lubensky model for smectic A liquid crystals is given by
\begin{eqnarray}
\mc{G_{C}} (\psi, \bn) &=&  \int_{\omm} \Big(D| \Dn^2 \psi |^2 + C_{\perp}| \Dn \psi |^2 + K |\nb \bn|^2
 \nonumber\\
 [-1.5ex] \label{CL-energy}\\[-1.5ex]
 && \qquad + \frac{g}{2} \Big( |\psi|^2 - \frac{r}{g} \Big)^2 - \chi_a  H^2(\bn \cdot \bh)^2 \Big) \, d\bx, \nonumber
\end{eqnarray}
where  $ \Dn= \nb - i q\mathbf{n},\; \Dn^2= \Dn \cdot \Dn, $ and $D,
C_{\perp},$ $ g$, and $r = T_{NA}-T$ are positive constants.
As for the de Gennes energy, we assume that the smectic order parameter $\psi(\bx) = \rho(\bx) e^{i q \omega(\bx)}$ has a constant density $\rho$. Under this assumption, the energy (\ref{CL-energy}) simplifies to
\begin{eqnarray*}
\mc{G_{C}} (\omega, \bn)\nt \nt &=& \nt \nt \int_{\omm} \Big(
 D \rho^2 q^4  | \nb \omega - \bn  |^4 + D \rho^2 q^2  (\Delta
\omega - \nb \cdot \bn)^2 + C_{\perp} q^2 \rho^2  | \nb \omega -
 \bn|^2  \\
  & & \qquad + K |\nb \bn|^2  - \chi_a  H^2 (\bn \cdot \bh)^2 \Big) \, d \bx.
\end{eqnarray*}
We again nondimensionalize with respect to the thickness of the sample by making the change of variables $\bar{\bx} = \bx/d$ in $\omm = (-L, L)^2 \times (-d, d) \subset \mathbb{R}^3$,  and obtain
\begin{eqnarray}
\mc{G_{C}} (\vp, \bn) &= &\frac{ d K}{ \e} \int_{\tilde{\omm}} \Big(  D_1 \e (\Delta \vp - \nb \cdot \bn)^2 + \frac{D_2}{2\e} |\nb \vp -\bn |^4  \nonumber\\
 [-1.5ex] \label{CL}\\[-1.5ex]
&& \qquad \quad  + \frac{1}{\e} |\nb \vp -\bn |^2  +\e |\nb \bn|^2 - \tau (\bn \cdot \bh)^2 \Big) \, d\bar{\bx}, \nonumber
\end{eqnarray}
where
\[
    D_1 = \fr{D \rho^2 q^2 }{K},\qquad   D_2 = \fr{2 D q^2}{C_{\perp}} , \qquad \vp = \frac\omega d,
\]
\[
\e = \frac{\lambda}{d}, \qquad \lambda = \sqrt{\frac{K}{C_{\perp} q^2 \rho^2}}, \qquad \tau = \frac{ \chi_a H^2 d^2 \e }{K},
\]
and
\[
\tilde{\omm} = (-l, l)^2 \times (-1,1 ), \qquad l = \frac{L}{d}.
\]
Note how the energy includes the same parameters $\e$ and $\lambda$ which appear in the de Gennes reformulation of section~\ref{section:deGennes}. The functional (\ref{CL}) has been studied for the onset of undulations in \cite{G-J3}, where the authors consider two boundary conditions for the layer variable $\vp$, and prove that the critical fields are estimated at $1$ and $\pi$, respectively. Here, as in section~\ref{section:deGennes} we consider larger values of $\tau$ by setting $\tau = \frac{\sigma}{\e}$ with $\sigma = \mathcal{O}(1)$, where the chevron structure is  seen experimentally.

We limit our study to the case $\bn \in \mathbb{S}^1$, and assume $\vp = z-g(x)$, where $g$ denotes the layer displacement. Hence, setting  $\bn = (\sin \theta, 0, \cos \theta)$ with $\theta = \theta(x)$,  and $I=(-l, l)$, the energy becomes
\begin{eqnarray*}
 &&\int_I \Big[  D_1 \e (g^{''} + \cos \theta \theta')^2 + \frac{D_2}{2 \e} \left( (g' + \sin \theta)^2 +( 1- \cos \theta)^2 \right)^2\\
&& \qquad  + \frac{1}{\e} \left((g'+\sin \theta)^2 + (1-\cos \theta)^2 \right) + \e \theta'^2 - \frac{\sigma}{\e} \sin^2 \theta \Big] \, dx,
\end{eqnarray*}
where again we have taken $\bh ={\bf e}_1$.
To highlight the double well structure of the potential, we rewrite  $1-\cos \theta = 2 \sin^2 \frac{\theta}{2}$ and add a constant. In conclusion, we are led to work with the energy
\begin{eqnarray}
F^{C}_{\e} (\theta, g) \nt \nt &=& \nt \nt \int_I \Big[  D_1 \e (g^{''} + \cos \theta \,  \theta')^2 + \frac{D_2}{2 \e} (g' + \sin \theta)^4 + \frac{1}{\e} (g'+\sin \theta)^2   \nonumber\\
 [-1.5ex] \label{CL_S1}\\[-1.5ex]
&& \qquad+ \frac{4 D_2}{\e} (g'+\sin \theta)^2 \sin^4 \frac{\theta}{2} +  \e\,  \theta'^2 +\frac{1}{\e} \, W(\theta) \Big] \, dx, \nonumber
\end{eqnarray}
where
\begin{equation}\label{well-CL}
W(\theta) = 8D_2\sin^8\frac{\theta}{2} + 4(1+\sigma) \sin^4 \frac{\theta}{2} - 4\sigma\sin^2 \frac{\theta}{2} +a_0.
\end{equation}
The constant $a_0$ is chosen so to ensure that $W(\theta)$ is nonnegative. Direct computations show that  for $\theta \in [-\pi, \pi]$, $W(\theta)$ presents a double well potential, and that, denoting the zero set of $W$ by $\{\pm \beta \}$ with $\beta>0$, one has $\beta = \alpha$ when $D_2 = 0$, and $\beta$ approaching $\alpha $ from the left as $D_2$ decreases to $0$. We provide some details on the behavior of the zeros of $W$ as function of $D_2$ and $\sigma$  in subsection~\ref{CL:well}. { In particular, from (\ref{expression-r}) and (\ref{expression-beta}) we see that $\sin \beta > 0 $ whenever $\sigma >0$. In Figure~\ref{fig:potential}, we present a plot of $W$ for fixed $D_2$ and various values of $\sigma$. 

\begin{figure}
\centering
        \epsfig{file=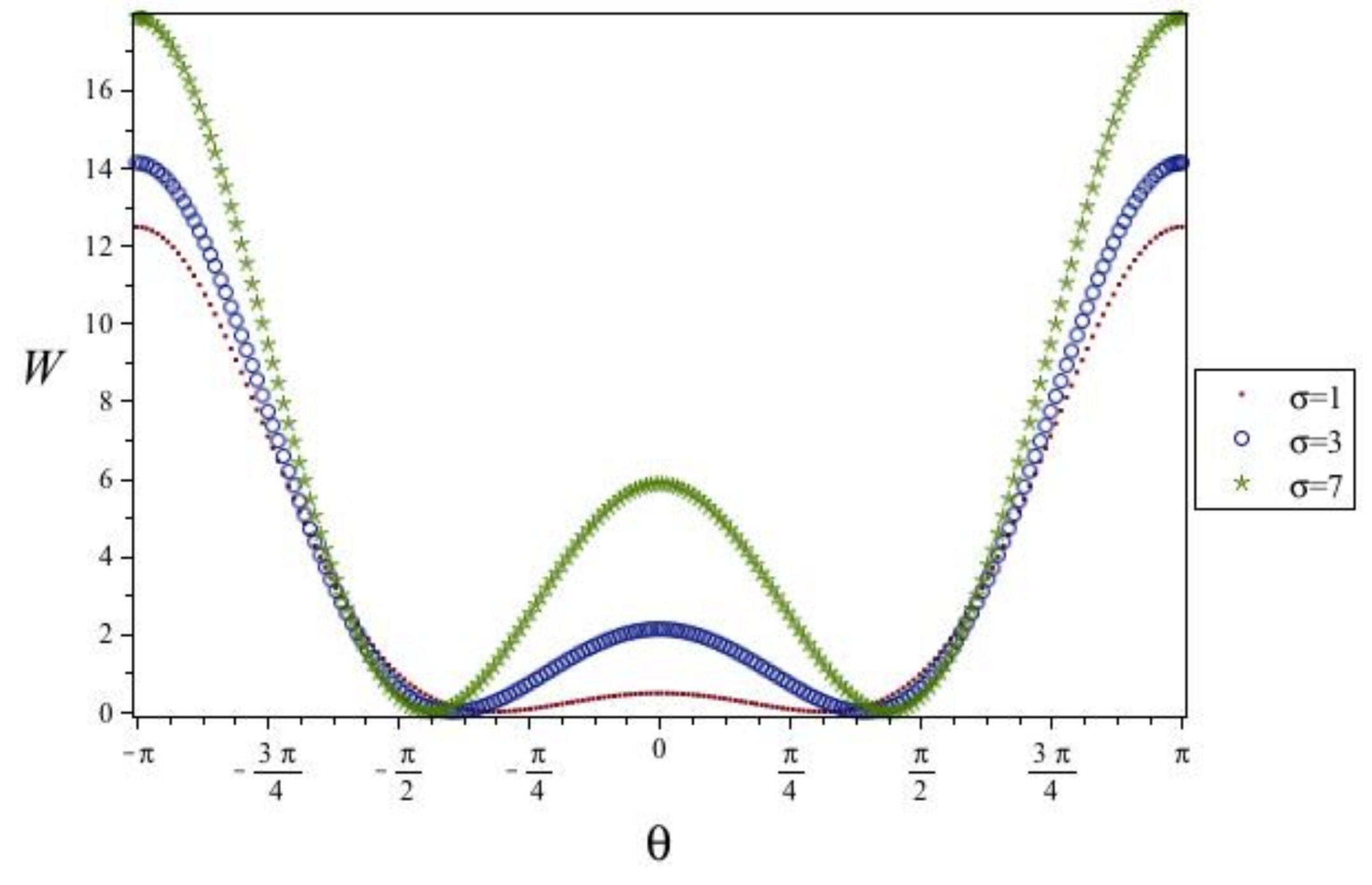,width=0.7\linewidth,clip=}
\caption{Plot of the double well potential $W(\theta)$ for $D_2=1$, and various values of $\sigma$. }
\label{fig:potential}
\end{figure}

We define the  sets
\begin{eqnarray}
\mathcal{Y_{C}} \nt \nt  &=& \nt \nt \{ (\theta, g) \in W^{1,2}(I) \times W^{2,2}(I) : \theta(-l) = \theta(l),  \nonumber \\
&& \qquad \qquad \qquad \qquad \qquad \qquad\;\; \, g(-l) = g(l), \; g'(-l) = g'(l) \}, \label{Y_C} \\
\mathcal{A_{C}} \nt \nt &=&  \nt \nt \{ (\theta, g) \in BV(I) \times W^{1,2}(I): g' \in BV(I), \nonumber\\
&& \qquad \qquad \qquad \qquad \qquad \qquad g'+\sin \theta =0 \, a.e,  \; \int_I \theta =0  \}, \label{A_C}
\end{eqnarray}
and the functionals
\begin{equation} \label{F_e_CL}
G^C_{\e} (\theta, g) :=
    \begin{cases}
         F_{\e} (\theta, g), \quad &\mbox{if } (\theta, g) \in  \mathcal{Y_C}  ,\\
         +\infty &\mbox{else},
    \end{cases}
\end{equation}
and
\begin{equation} \label{F_0_CL}
G^C_0 (\theta, g) :=
    \begin{cases}
         \int_I | (\Phi \circ \theta)' | + |\Phi(\ttheta(l))- \Phi(\ttheta(-l))| \quad &\mbox{if } (\theta, g)   \in \mc{A_C} , \\
         +\infty &\mbox{else},
    \end{cases}
\end{equation}
 in here $\Phi(s) =  2 \int_{-\beta}^s \sqrt{W(t)}\, dt$,  while $\ttheta(\pm l)$ denotes the trace of $\theta$ on $\pm l$. Note that for $(\theta, g) \in \mc{A_C}$,
\begin{equation} \label{jumps_CL}
\int_I | (\Phi \circ \theta)' | = \Phi(\beta) ( \mbox{number of jumps} ).
\end{equation}
A $\Gamma$-convergence analysis for the functional $G^C_\e$ as $\e$ tends to zero results in a picture analogous to the one obtained for the de Gennes functional $G_\epsilon$  in section~\ref{section:deGennes}. This shows that the de Gennes model captures the essence of the chevron creation phenomenon seen in experiments.
In particular, we have the following result:

\begin{theorem}\label{thm:convergence_CL}
Let $\{(\theta_j, g_j)\} \in \mc{Y_C}$ be a sequence of minimizers to $G^C_{\e_j}$ for $\e_j \to 0$. Then there are sequences of numbers $ \{a_j\} \subset \{\pm 1\}$   and $\{c_j\} \subset (-l, l) $ such that $(\hat{\theta}_j, \hat{g}_j)  \to (\theta, g) $ in $L^1(I) \times L^2(I)$ where
\begin{eqnarray*}
&&\hat{\theta}_j (x) = a_j \theta_j(x+c_j), \qquad \hat{g}_j (x) = a_j g_j(x+c_j),\\
&&\theta = \beta \chi_J - \beta (1-\chi_{J}), \qquad \mbox{where} \quad J = (-l/2, l/2).
\end{eqnarray*}
Furthermore, $G_0(\theta, g) = 2 \Phi(\beta).$
\end{theorem}

\begin{proof}
By following the proof  of Theorem~\ref{thm:deGennes}, the result is a consequence of Proposition \ref{CL-compactness}, Lemmas~\ref{thm:lower_CL}~and~\ref{thm:upper_CL} below, and the uniqueness of the minimizing pattern of $G^C_0$. By definition of $\mc{A_C}$ a minimizer  $(\theta,g)$ of $G^C_0$ must have $\int_I \theta =0$, but since the minimizer satisfies $\ttheta(-l) = \ttheta(l) $, it follows that $\theta$ has  two internal jumps.
\end{proof}

As previously done, we start by providing a compactness result, then proceed to obtain the lower and upper bounds of Lemmas~\ref{thm:lower_CL}~and~\ref{thm:upper_CL}. The arguments of the proofs in this section are combinations of ideas from the standard Modica-Mortola references  \cite{Modica1987,Modica-Mortola}, and the classical work of Owen, Rubinstein and Sternberg \cite{Sternberg90}, which illustrates how to treat boundary conditions under  $\Gamma$-limits.

\begin{proposition}  \it{(Compactness)}\label{CL-compactness}
Let the sequences $\{\e_j\}_{j \uparrow \infty} \subset (0, \infty)$, and $ \{ (\theta_j , g_j)\}_{j \uparrow \infty}~ \subset~ \mc{Y_C}$, be such that
\[
\e_j \to 0, \qquad \mbox { and } \qquad \{ G^C_{\e_j}(\theta_j, g_j)\}_{j \uparrow \infty} \qquad \mbox{is bounded.}
\]
There exists a subsequence $\{ (\theta_{j_k}, g_{j_k}) \}$ and a $(\theta, g) \in \mathcal{A_C}$ such that
\[
\theta_{j_k} \to \theta \mbox{ in } L^1(I) \qquad \mbox{ and } \qquad g_{j_k} - \frac{1}{|I|} \int_I g_{j_k}\, dx \to g \mbox{ in } L^2(I).
\]

\end{proposition}

\begin{proof}
Since $G_{\e_j} (\theta_j, g_j) \leq M$, we may assume that $(\theta_j, g_j) \in \mathcal{Y_C}$, so that
\[
2\, \int_I \sqrt{W(\theta)} \,  |\theta_j'| \, dx \leq M,
\]
which leads to
\[
\int_I |(\Phi (\theta_j) )' | \, dx \leq M,
\]
where $\Phi(s) =  2 \int_{-\beta}^s \sqrt{W(t)}\, dt$. This, together with $ |\theta_j| \leq \pi$, implies that $\Phi \circ \theta_j $ is uniformly bounded in $W^{1,1}(I)$. Hence, we may extract a subsequence, not relabeled, such that
\[
\Phi \circ \theta_j \to  w \mbox{ in } L^1(I),
\]
for some $w \in BV(I)$, and, possibly up to another subsequence,  a.e.  in $I$.
Being $\Phi$ continuous and strictly increasing, we also have $\theta_j \to \Phi^{-1}(w) =: \theta $ a.e., and given that $W(\theta_j)$ converges to zero a.e, this implies $\theta(x) \in \{\pm \beta\}$ a.e. in $I$. Then $w = \Phi(\beta) \chi_E + \Phi(-\beta) (1-\chi_E)$, from which we gather
\begin{equation} \label{bv-theta}
\theta = \beta \chi_E - \beta(1 - \chi_E),
\end{equation}
and since $w\in BV(I)$, we conclude  $\theta \in BV(I)$.

From the energy bound, we also know $\int_I (g_j'+\sin \theta_j)^2 \, dx \leq M \e_j$, which yields $\int_I |g_j'|^2 \, dx \leq C$. Thus there exist a subsequence (not relabeled), and a function $g\in W^{1,2}(I)$ such that
\[
 g_j- \frac{1}{2l} \int_I g_j\, dx \rightharpoonup g \mbox{ in } W^{1,2}(I).
\]
But
\begin{equation} \label{g_functional}
0 =  \liminf_{j \to \infty} \int_I (g_j'+ \sin \theta_j )^2 \, dx \geq \int_I (g' + \sin \theta)^2 \, dx,
\end{equation}
gives
$
g' + \sin \theta =0
$ a.e in $I$.
Furthermore, the periodic boundary condition $g_j(-l) = g_j(l)$ and (\ref{g_functional}) imply $\int_I \sin \theta_j(x) \, dx \to 0$, hence using the Dominated Convergence theorem  we have $\int_I \sin \theta =0 $, and conclude  $\int_I \theta =0$ by (\ref{bv-theta}).
\end{proof}

\begin{lemma} \it{(Lower semi-continuity)} \label{thm:lower_CL}
For every $(\theta, g) \in L^1(I) \times L^2(I)$ and every sequence $(\theta_j, g_j) \in \mathcal{Y_C} $ such that $(\theta_j, g_j)$ converges to $(\theta, g) $ in $L^1(I)\times L^2(I)$ there holds
\[
       \liminf_{j \to \infty} G^C_{\e_j} (\theta_j, g_j) \geq G^C_0(\theta, g),
\]
and $
(\theta, g)  \in \mc{A_C}.$
\end{lemma}

\begin{proof}
If $(\theta_j, g_j) \notin \mc{Y}$, then $G^C_{\e_j} (\theta_j, g_j) = \infty$ and the inequality is trivial. Therefore, we may consider $(\theta_j, j_j) \in \mc{Y_C}$ and $G^C_{\e_j} (\theta_j, g_j) \leq M$ for some constant $M$. By Proposition \ref{CL-compactness}, we may assume that $(\theta, g) \in \mc{A_C}$, i.e., $\theta \in BV(I; \{\pm \beta \}).$
The first term in the $\Gamma$-limit is the essential feature in the Modica-Mortola model. The second term arises due to the periodic boundary condition. The mass constraint is compatible with the $\Gamma$-limit, however, the boundary value is not compatible with the $\Gamma$-limit. By adding this term we may pass the periodic boundary condition to the $\Gamma$-limit. The proof is motivated by \cite{Sternberg90}.

Set $I_{\delta} = (-l-\delta  , l+ \delta)$ and extend $\theta_j$ and $\theta$ on $I_{\delta}$ by periodicity on $I_{\delta} - \bar{I}$, that is consider
$$
\htheta_j (x)= \begin{cases}
	  \theta_j(x+2l) \qquad & \mbox{if } x\in (-l-\delta,-l) \\
           \theta_j(x) \qquad &\mbox{if } x\in {I},\\
           \theta_j(x-2l) \qquad &\mbox{if } x\in (l,l+\delta) ,
           \end{cases}
$$
and
$$
\htheta (x)= \begin{cases}
	  \theta(x+2l) \qquad & \mbox{if } x\in (-l-\delta,-l) \\
           \theta(x) \qquad &\mbox{if } x\in {I},\\
           \theta(x-2l) \qquad &\mbox{if } x\in (l,l+\delta).
           \end{cases}
$$
Since $\theta \in BV(I_{\delta})$, the trace of $\theta$ at $\pm l$ can be defined, see \cite{giusti1984minimal}, as $\ttheta(l)\equiv \theta^-(l)$, where
\begin{equation}\label{thetaminus}
\theta^-(l)=\lim_{\rho \to 0^+} \frac 1\rho \int_{l-\rho}^l \theta(s) \, ds,
\end{equation}
and $\ttheta(-l)\equiv \theta^+(-l)$, where
\begin{equation}\label{thetaplus}
\theta^+(-l)=\lim_{\rho \to 0^+} \frac 1\rho \int_{-l}^{-l+\rho}\theta(s) \, ds.
\end{equation}
Hence, we have $\htheta \in BV(I_{\delta})$,  $\htheta_j$ converges to $\htheta$ in $L^1(I_{\delta})$, and
$$
\htheta^+(-l)=\ttheta(-l),\quad \htheta^-(-l)=\ttheta(l), \quad \htheta^+(l)=\ttheta(-l), \quad \htheta^-(l)=\ttheta(l).
$$
Additionally, using the fact that $\theta_j \in W^{1,2}(I)$ and $\theta_j(-l) = \theta_j(l)$, we see that
$$
\htheta_j^+(l)=\htheta_j^-(l)=\htheta_j^+(-l)=\htheta_j^-(-l).
$$
Therefore, for $0<\delta \leq l$ it holds
\begin{eqnarray*}
&&\, \nt \nt \nt \nt \nt \nt \nt \nt\nt \nt \nt \nt 2 \,  \liminf_{j \to \infty} G^C_{\e_j} (\theta_j, g_j) \geq \liminf_{j \to \infty} \left [ \int_{I} | (\Phi \circ \theta_j )'| +  \int_{I_\delta \setminus {\bar I}} | (\Phi \circ \theta_j )'|  \right] \qquad \\
    &=& \liminf_{j \to \infty} \left[ \int_{I_{\delta}} | (\Phi \circ \htheta_j )'| + \int_{I_{\delta} \setminus\bar{I}} | (\Phi \circ \htheta_j )'| \right. \\
    &\quad& \qquad \left.  + |\Phi(\htheta_j^+(r))- \Phi(\htheta_j^-(r))|         +    |\Phi(\htheta_j^+(-r))- \Phi(\htheta_j^-(-r))|               \right]\\
     &=& \liminf_{j \to \infty}  \int_{I_{\delta}} | (\Phi \circ \htheta_j )'| \\
      & \geq &  \int_{I_{\delta}} | (\Phi \circ \htheta )'|  =  \int_{I} | (\Phi \circ \theta )'|+ \int_{I_{\delta} \setminus\bar{I}} | (\Phi \circ \htheta)'|   \\
      & \quad &  \qquad +|\Phi(\htheta^+(l))- \Phi(\htheta^-(l))| +  |\Phi(\htheta^+(-l))- \Phi(\htheta^-(-l))|  \\
      &=& \int_{I} | (\Phi \circ \theta )'|+ \int_{I_{\delta} \setminus\bar{I}} | (\Phi \circ \htheta)'|   + 2\,  |\Phi(\ttheta(l))- \Phi(\ttheta(-l))|,
\end{eqnarray*}
and if we take $\delta= l$ we have
$
\, \displaystyle{2  \liminf_{j \to \infty} G^C_{\e_j} (\theta_j, g_j) \geq 2\,  G^C_0(\theta, g).}
$
The feature (\ref{jumps_CL}) can be obtained by the same proof of the lower bound for the Modica-Mortola model: From $\theta \in BV(I;\pm \{\beta\})$, we may assume that $\theta$ has internal jumps at $t_j$ for $j = 1, 2, \cdots, m$. Then, for $\delta>0$ small enough, it holds
\[
G^C_{\e_j} (\theta_j, g_j) \geq \sum_{i=1}^m \int_{t_i-\delta}^{t_i+ \delta} e_{\e_j}(\theta_j, g_j)\, dx 
\]
where $e_{\e_j}$ is the integral density of $F^C_{\e_j}$. We consider one term, assuming $t_i=0$ and $\theta_j(\pm\delta) \to \pm \beta$. Making the change of  variables, $\bar{x} = \frac{x}{\e_j}, \bar{\theta}_j(\bar{x}) = \theta_j(x), \bar{g}_j(\bar{x}) = g_j(x)$, we have
\begin{eqnarray*}
&&\int_{-\delta}^{ \delta} e_{\e_j}(\theta_j, g_j)\, dx \geq 2\, \int_{-\frac{\delta}{\e_j}}^{\frac{\delta}{\e_j}} \sqrt{W(\bar{\theta}_j)} \,  |\bar{\theta}_j'| \, d \bar{x} \geq 2 \int_{\ttheta_j(-\delta/\e_j)}^{\ttheta_j(\delta/\e_j)} \sqrt{W(t)} \,  dt\\
&& = 2 \int_{\theta_j(-\delta)}^{\theta_j(\delta)} \sqrt{W(t)} \,  dt \longrightarrow 2\int_{-\beta}^{\beta}\sqrt{W(t)}\, dt = \Phi(\beta).
\end{eqnarray*}
\end{proof}
Next, we derive the upper bound inequality.

\begin{lemma} \it{(Construction)} \label{thm:upper_CL}
For any $(\theta, g) \in L^1(I) \times L^2(I)$, there exists a sequence $(\theta_j, g_j) \in \mathcal{Y_C}  $, converging in $L^1(I) \times L^2(I) $ as $j \to \infty$, to $(\theta, g)$, and such that
\[
  \limsup_{j \to \infty} G^C_{\e_j} (\theta_j,g_j) \leq G^C_0(\theta, g).
\]
\end{lemma}

\begin{proof}
If $G^C_0(\theta, g) = \infty$, the result is trivial. Assume that $(\theta, g) \in \mc{A_C}$.
The sequence $\theta_j$ is obtained via the well-known Modica-Mortola construction, see \cite{Modica1987}, hence in the following we will omit the details and provide only the essential modifications.
If $\ttheta(l)=\ttheta(-l)$, then there are only internal transitions, which given the condition $\int_I \theta(s) \, ds =0$ it means there are at least two of them. We follow \cite{Modica1987}, and introduce the set $A=\{t\in I : \theta(t)=-\beta\}$, as well as the functions
\begin{eqnarray}\label{distance_CL}
h (x)= \begin{cases}
	 - \mbox{dist}(x,\partial A) \qquad & \mbox{if } x\in A \\
          \quad \mbox{dist}(x,\partial A) \qquad & \mbox{if } x\notin  A,
           \end{cases}
\end{eqnarray}
and
\begin{eqnarray}\label{jump}
\chi_0(t)= \begin{cases}
	 - \beta \qquad & \mbox{if } t<0 \\
          \quad \beta \qquad & \mbox{if } t\geq 0.
           \end{cases}
\end{eqnarray}
Note that  around a jump we have $\theta(x)=\chi_0(h(x))$. The next step is to consider
\begin{equation}
\psi_{\e}=\int_{-\beta}^t \left( \frac{\e^2}{ \e+ W(s)}\right)^{\frac 12} \, ds,
\end{equation}
which  has a well-defined inverse function $\phi_\e:[0,\eta_\e] \to [-\beta,\beta]$, here  $\eta_\e=\psi_{\e}(\beta) \leq 2\e^{1/2} \beta$, and that can be smoothly extended outside the interval $[0,\eta_{\e}]$ to $-\beta $ for $t<0$, and $\beta$ if $t\geq \eta_\e$. By construction, for every $t$ we have $\phi_\e(t)\leq \chi_0(t)$, and $ \phi_\e(t+\eta_\e)\geq \chi_0(t)$, and since $0<\beta<\frac \pi 2$ (see subsection~\ref{CL:well}), we also have $\sin(\phi_\e(t))\leq \sin(\chi_0(t))$ and $\sin(\phi_\e(t+\eta_\e))\geq \sin(\chi_0(t))$, Therefore, we can find a $\delta_\e\in [0,\eta_\e]$ such that
\begin{equation}
\int_I \sin(\phi_\e(h(x)+\delta_\e) ) \, dx = \int_I \sin(\chi_0(h(x))) \, ds.
\end{equation}

Now, using this construction around each transition point $t_i$ of $\theta$, because $\eta_\e \leq 2\e^{1/2} \beta$ and $\int_I \sin(\theta(s)) \, ds=0$,  for $\e_j$ small enough, we can obtain a sequence of $\theta_j$ with
\begin{equation*}
\ttheta_j(-l)=\ttheta_j(r)\,  \mbox{ and } \,  \int_I \sin(\theta_j(s)) \, ds =0,
\end{equation*}
and which converges to $\theta$ in $L^1(I)$. Additionally, for every $i$, we have
\[
\limsup_{j \to \infty} \int_{t_i-\delta_{\e_j}}^{t_i+ \eta_{\e_j}- \delta_{\e_j}} \left(\e_j \theta_j'^2 + \frac{1}{ \e_j} W(\theta_j ) \right) \leq 2\int_{-\beta}^{\beta} W(t) \, dt = \Phi(\beta).
\]
If $\ttheta(l)\neq \ttheta(-l)$, say $\ttheta(-l)=-\beta$, a boundary layer either at $l$ or $-l$ must occur due to the periodic boundary condition for $\theta_j$, and the construction requires some small changes. By the definition of trace given in (\ref{thetaminus})-(\ref{thetaplus}), we can find a small $\delta$ such that $\ttheta(-l+\delta)= -\beta$, and there are no jumps of $\theta$ in $(-l, -l+\delta)$. Consider the function
\begin{eqnarray}\label{translate}
\theta_\delta(t)= \begin{cases}
	 \theta (t)\qquad & \mbox{if } -l+\delta<t<l \\
           \theta(t-2l) \qquad & \mbox{if } l<t<l+\delta,
           \end{cases}
\end{eqnarray}
and repeat the previous construction for $\theta_\delta$ on $(-l+\delta, l+\delta)$, calling the corresponding functions $(\theta_\delta)_j$, we have that
\begin{eqnarray}\label{translateback}
\theta_j(t)= \begin{cases}
	   (\theta_\delta)_j(t+2l) \qquad & \mbox{if } -l<t<-l+\delta \\
         (\theta_\delta)_j(t)\ \qquad & \mbox{if } -l+\delta<t<l,
           \end{cases}
\end{eqnarray}
belongs to $\mathcal{Y_C}$, for $\e_j>0$ small enough, and
\begin{eqnarray*}
&&\limsup_{j \to \infty} \int_I\left(\e \theta_j'^2 + \frac{1}{ \e} W(\theta_j ) \right) = \limsup_{j \to \infty} \int_{-l+\delta}^{l+\delta} \left(\e (\theta_\delta)_j'^2 + \frac{1}{\e} W(\theta_\delta)_j  \right)\\
 && \qquad  \leq  \Phi(\beta) ( \mbox{number of interior jumps of $\theta$})+ \Phi(\beta)\\
&& \qquad = \Phi(\beta) (\mbox{number of interior jumps of $\theta$}) + |\Phi(-\beta) - \Phi(\beta)| = G^C_0(\theta, g)
\end{eqnarray*}
Finally, since $\theta_j \in W^{1,2}(I)$, we define
\[
g_j(x) = - \int_{-l} ^x \sin \theta_j(t) \, dt,
\]
then {$g_j \in W^{2,2}(I)$}, $g_j(-l)=g_j(l)$, and for the sequence $(\theta_j,g_j)$ we have the desired upper bound inequality.
\end{proof}

\begin{remark}\label{Chen-Lub_remark}
We note that the analysis for the flat torus of section \ref{section:deGennes} can be also applied to the Chen-Lubensky energy for $\bn$ over $\mathbb{S}^2$ in two dimensions. Setting $\vp = z - g(x,y)$, $\bn = \bn(x,y)$ and $\bn_{\parallel} = (n_1, n_2)$, the energy (\ref{CL}) becomes
\begin{eqnarray*}
 &&\int_{\omm} \Big(  D_1 \e (\Delta g + \nb \cdot \bn)^2 + \frac{D_2}{2 \e} \left(|\nb g +\bn_{\parallel} |^4 + 2 (1-n_3)^2 |\nb g + \bn_{\parallel}|^2 \right)  \nonumber\\
&& \qquad \quad  + \frac{1}{\e} |\nb g +\bn_{\parallel} |^2  +\e |\nb \bn|^2 + W(\bn) \Big) \, dx dy, \nonumber
\end{eqnarray*}
with potential $W$ given by
\[
W(\bn ) =\frac{ D_2}{2}(1-n_3)^4 + (1-n_3)^2 + \sigma (n_3^2 + n_2^2) + b_0.
\]
Here, as in the one dimensional case, $b_0$ can be chosen so to ensure that $W$ is nonnegative. Following the calculations of subsection \ref{CL:well}, we can see that $W(\bn)$ is also a double well potential with two zeros $\bn^{\pm} = (\pm \bar{n}_1, 0, B)$, where
\begin{equation}\label{B_Chen}
B = 1 - 2 \sqrt{\frac{1+ \sigma}{3 D_2}} \sinh \left[\frac{1}{3} \arsinh \left( \frac{3 \sigma}{2(1+ \sigma)}\sqrt{\frac{3 D_2}{1+\sigma}}\right) \right]
\end{equation}
and $B(\sigma, D_2) \to A = (1+\sigma)^{-1}$ as $D_2 \to 0^+$. Then, the same proofs of section \ref{section:deGennes} give that the $\Gamma$-limit (\ref{G_0}) established for the de Gennes  energy is also the $\Gamma$-limit of the Chen-Lubensky model, with $A=(1+\sigma)^{-1}$ replaced by $B$.

\end{remark}

\section{Numerical Simulations} \label{section:numerics}
We consider the gradient flow (in $L^2)$ of the energy (\ref{energy-numerics}) (up to a multiplicative constant):
\[
 F(\bn, \phi) = \frac{1}{2} \int_{\Omega} \left \{ \frac{1}{\e} |\nabla \phi-\bn|^2+\e |\nabla \bn|^2 - \tau (\bn \cdot \bh)^2\right \},
\]
where $\omm = (-l, l)^2 \times (-1, 1) $ and study the behavior of the solutions with Dirichlet boundary conditions for both $\bn$ and $\phi$ on the top and bottom plates. For the sawtooth undulation, periodic boundary conditions are imposed for both $\bn$ and $\phi$ in the $x$ and $y$ directions. The gradient flow equations are
\begin{eqnarray}
\fr{\pa \phi}{\pa t} &=& \frac{1}{\e} \left( \Delta \phi - \nb \cdot \bn \right),  \nonumber\\
 [-1.5ex] \label{GradientFlow}\\[-1.5ex]
\fr{\pa \bn}{\pa t}  &=& \Pi_{\bn} \left( \e \Delta \bn + \frac{1}{\e} \left(\nb \phi - \bn\right) + \tau (\bn \cdot \bh) \bh \right), \nonumber
\end{eqnarray}
where we have defined, for a given vector $f \in \mathbb{R}^3$, the orthogonal projection onto the plane orthogonal to the vector $\bn$ as
\[
 \Pi_n(f) = f-(\bn \cdot f) \bn.
\]
This projection appears as a result of the constraint $\bn \in \mathbb{S}^2$.

As initial condition, we consider a small perturbation from the undeformed state. More precisely, for all $(x,y,z) \in \omm$,
\begin{eqnarray*}
\bn (x,y,z,0) & = &\frac{(\epsilon u_1, \epsilon u_2, 1+ \epsilon u_3)}{|(\epsilon u_1, \epsilon u_2, 1+ \epsilon u_3) |}, \\
\phi (x,y,z,0) &= & z + \epsilon \phi_0,
\end{eqnarray*}
where a small number $\epsilon = 0.1$, $u_1, u_2,u_3$ and $ \phi_0$ are arbitrarily chosen.
We impose strong anchoring condition for the director field, and a Dirichlet boundary condition on $\phi$ at the top and the bottom plates, that is
\[
\bn(x,y,\pm 1, t)=  \mathbf{e}_3, \quad \mbox{and} \quad \phi(x,y,\pm 1, t) = z,
\]
for all t.

 We use a Fourier spectral discretization in the $x$ and $y$ directions, and second order finite differences in the $z$ direction. The fast Fourier transform is computed using the FFTW libraries \cite{F-J}. For the temporal discretization, we combine a projection method for the variable $\bn$ \cite{E-W}, with a semi-implicit scheme for $\phi$. We take $l =4, \e = 0.2$ and 128 gridpoints in the $x$ and $y$ direction, which ensures that the transition layers are accurately resolved.

We have solved this system in \cite{G-J2} for the study of layer undulation phenomena in a two dimensional domain, $\omm = (-l,l) \times (-1,1)$ and $\bn \in \mathbb{S}^1$. We also proved that the layer undulation occurs at $\tau = \mathcal{O}(1)$ as the first instability in \cite{G-J2}. In fact, the critical fields for undulational instability are $\pi$ and $1$, when Dirichlet and natural boundary conditions are imposed on $\phi$ at $z= \pm 1$, respectively. Here we consider a three dimensional domain with $\bn \in \mathbb{S}^2$. One can show that the same estimate of the critical field and description of the layer undulations can be obtained for the three dimensional case. More detailed analysis with various magnetic fields in a three dimensional domain will appear in a future publication. In Figure \ref{fig:undulations3D} we illustrate the formation of layer undulations, and confirm that the layer undulations occur at $\tau = \pi$. Numerical simulations show that the undeformed state ($\bn = \mathbf{e}_3, \, \phi= z$) is an equilibrium state at $\tau = 3$ and undulations appear at $\tau = 3.2$. Then the sinusoidal oscillation transforms into a chevron structure at much stronger fields, as shown in Figure \ref{fig:chevron3D}.

 In Figure \ref{fig:chevron3D} we depict the configurations of each component of $\bn$ and surface of $\phi$ in the middle of the domain, $z=0$. The pictures clearly show the zigzag pattern of the director. The directors and the layers are illustrated with various field strengths in Figures \ref{fig:undulations3D} and \ref{fig:chevron3D}. {One can notice from Figure~ \ref{fig:chevron3D} that the transition paths connecting $\bn^{+}$ and $\bn^-$ do not depend on $n_2$. This is consistent with the explicit form (\ref{geodesic-curve}) of the d-geodesic curve, which is introduced and proved in appendix~\ref{geodesics}.} It also indicates that the period becomes larger as the field strength increases. In fact, the numerical simulations show that the increase of the wavelength occurs simultaneously with the evolution of the chevron pattern, which is also observed in the experiment \cite{ishikawa2001defects, S-S-L}.

\begin{figure}
\centering
    \begin{tabular}{cc}
       \epsfig{file=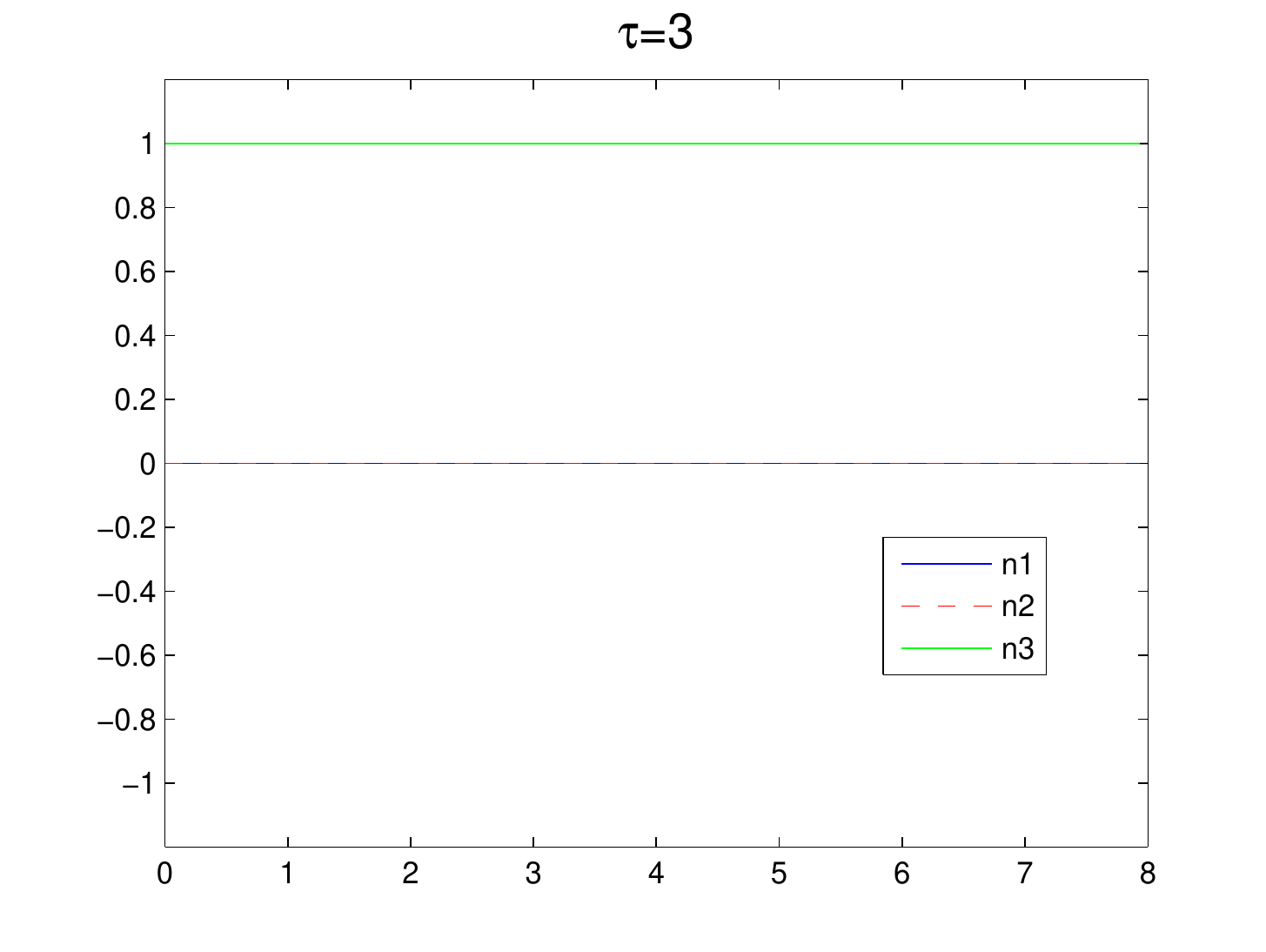,width=0.44\linewidth,clip=} \label{director_h3}&
       \epsfig{file=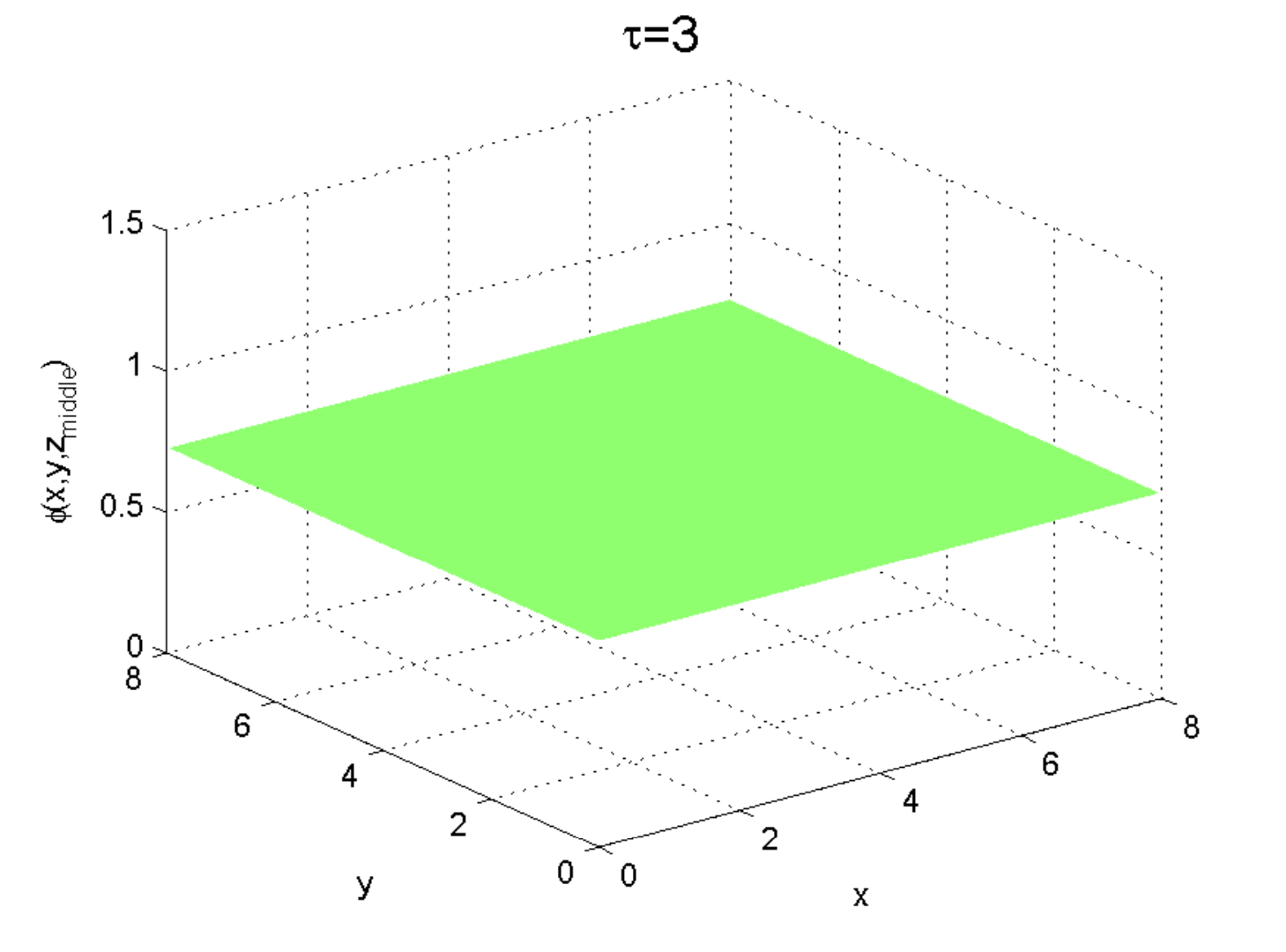,width=0.44\linewidth,clip=} \label{surface_h3}\\
        \epsfig{file=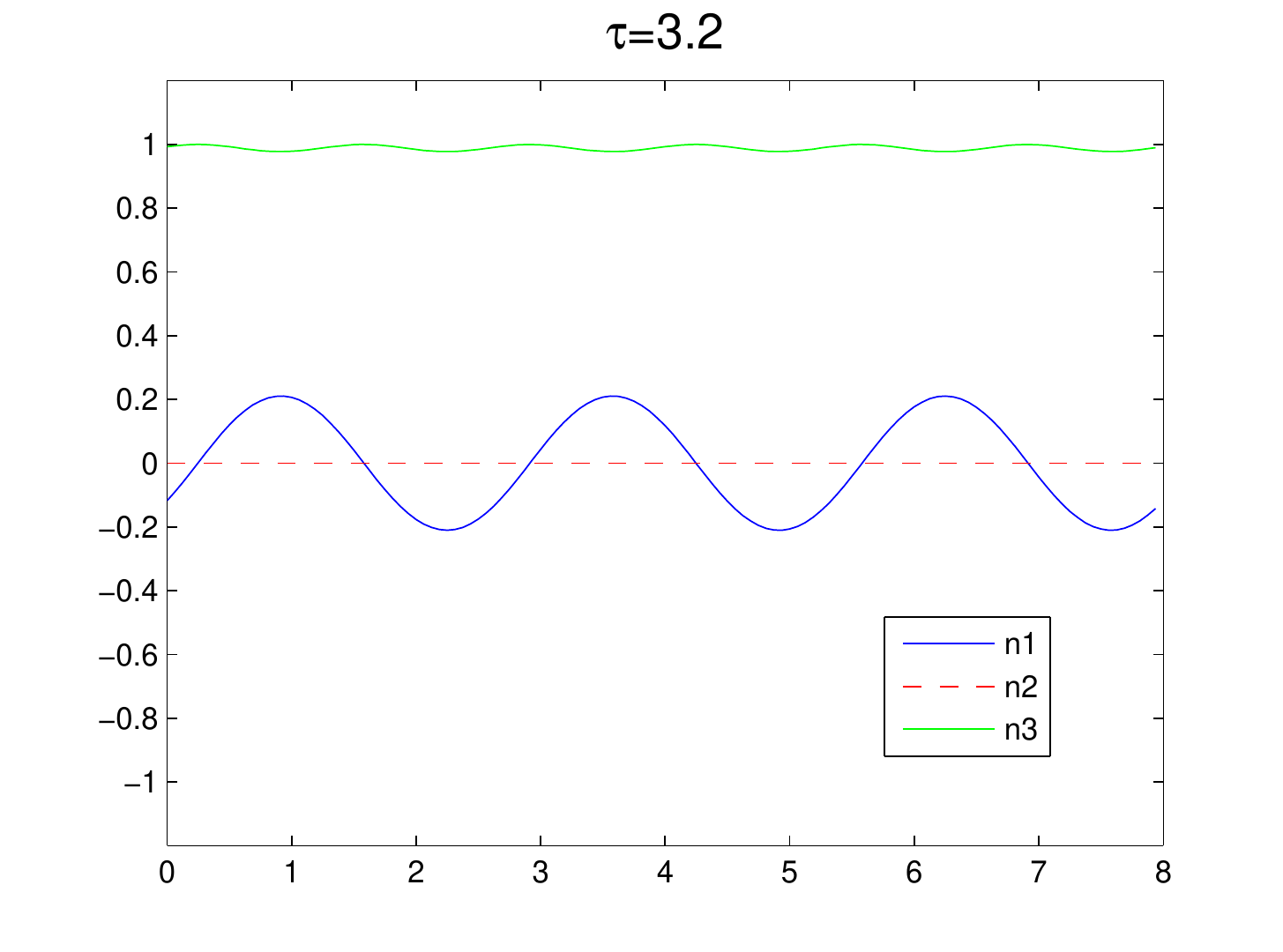,width=0.44\linewidth,clip=} \label{director_h3dot2}&
        \epsfig{file=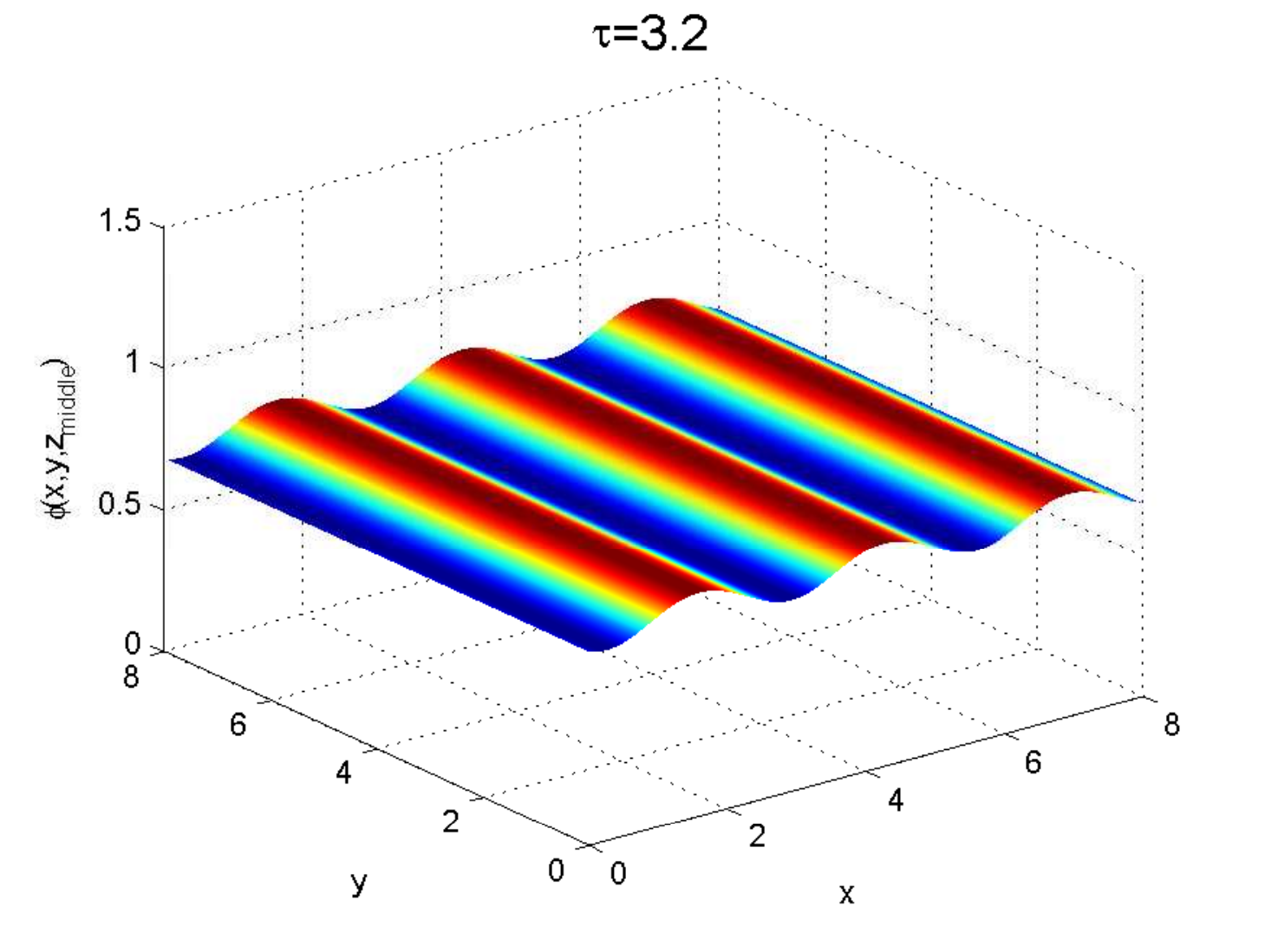,width=0.44\linewidth,clip=} \label{surface_h3dot2}\\
        \epsfig{file=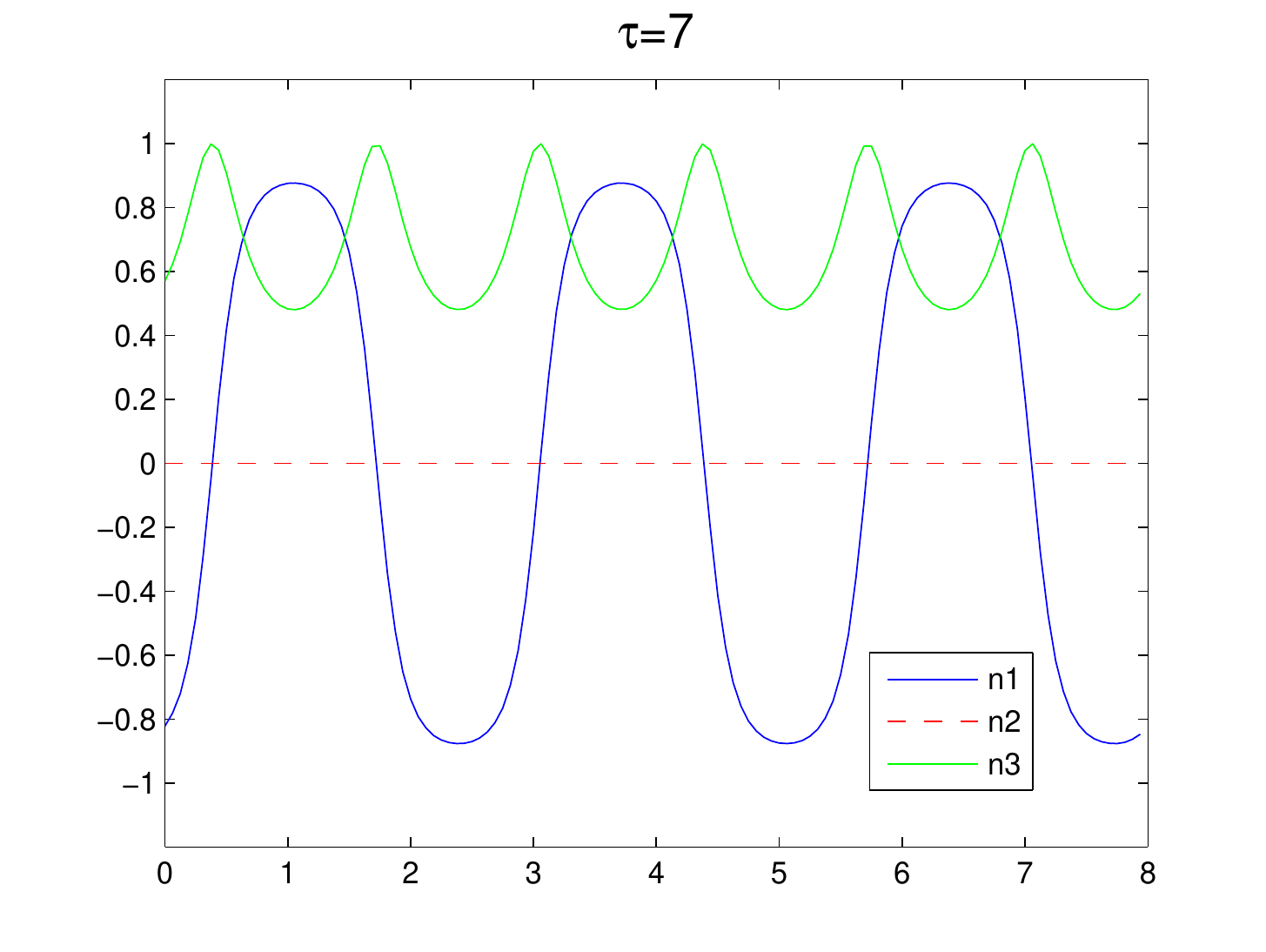,width=0.44\linewidth,clip=} \label{director_h7}&
       \epsfig{file=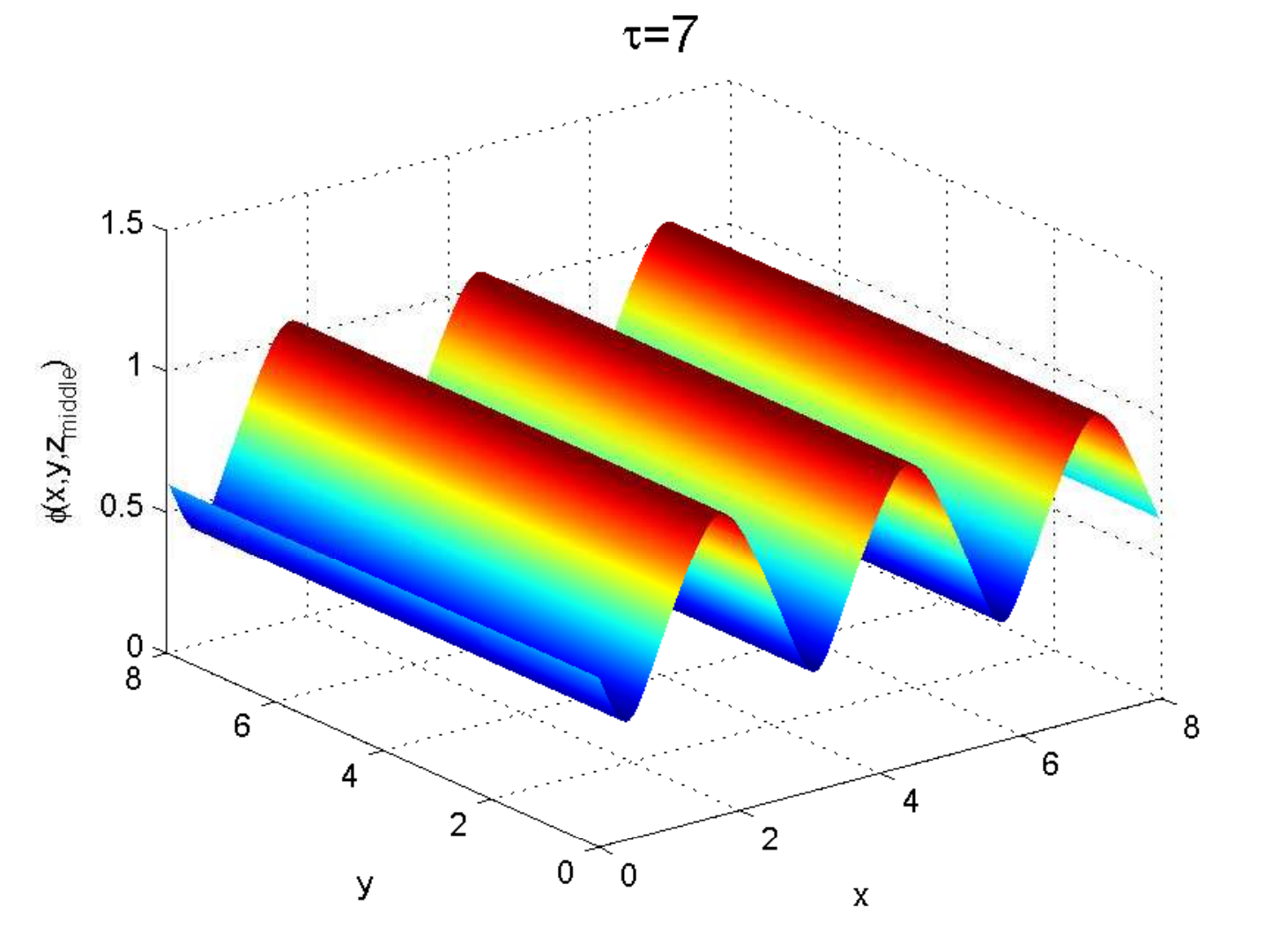,width=0.44\linewidth,clip=} \label{surface_h7}\\
    \end{tabular}
    \caption{Undulations: Numerical solution of (\ref{GradientFlow}) with strong anchoring conditions on the bounding plates. The first and second columns depict scalar components of directors and surface of the layer in the middle of the cell, respectively. The magnetic field strength $\tau = 3, 3.2, 7$ for each row.}
    \label{fig:undulations3D}
\end{figure}

\begin{figure}
\centering
    \begin{tabular}{cc}
        \epsfig{file= 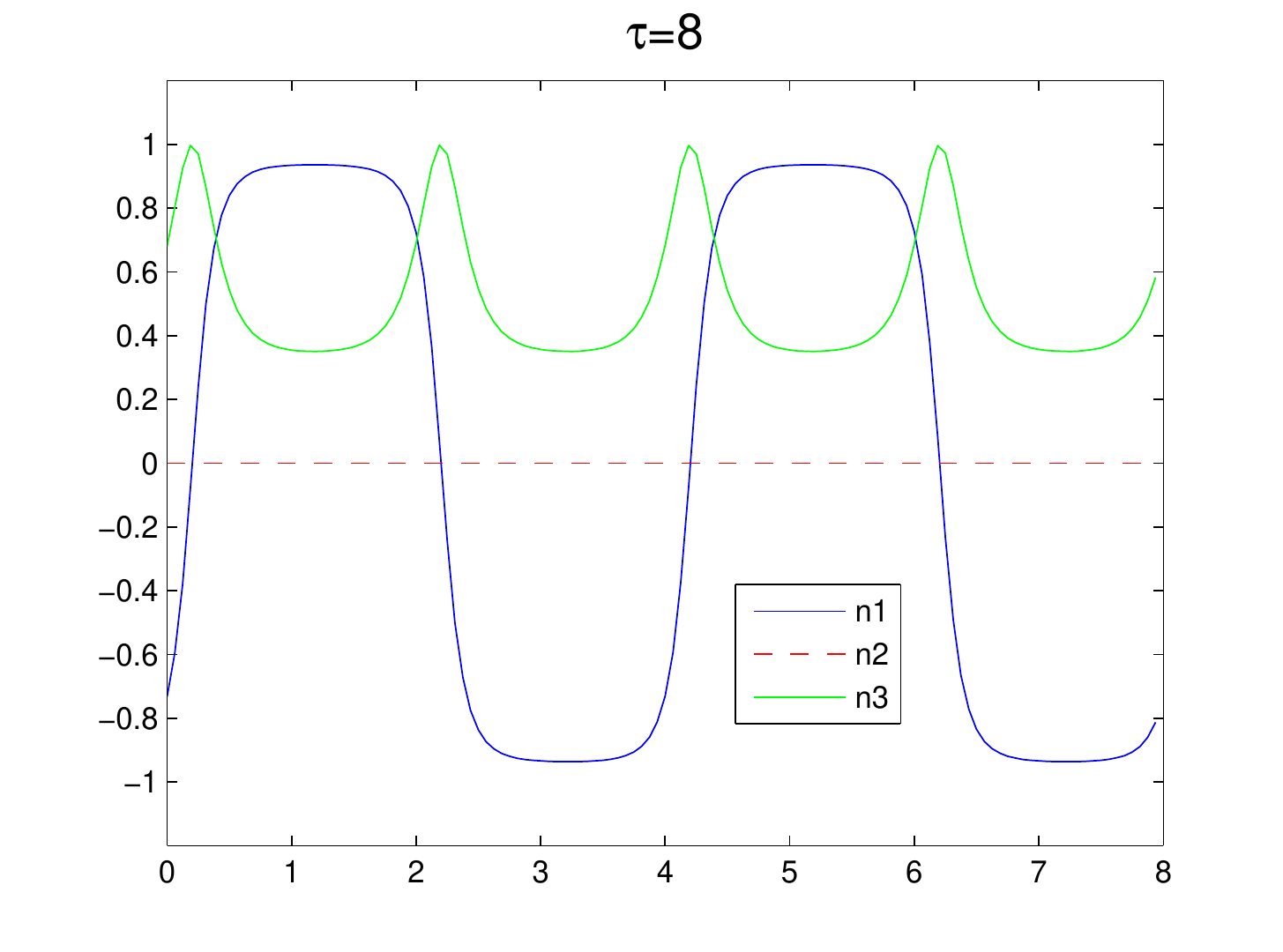,width=0.44\linewidth,clip=} \label{director_h8}&
         \epsfig{file=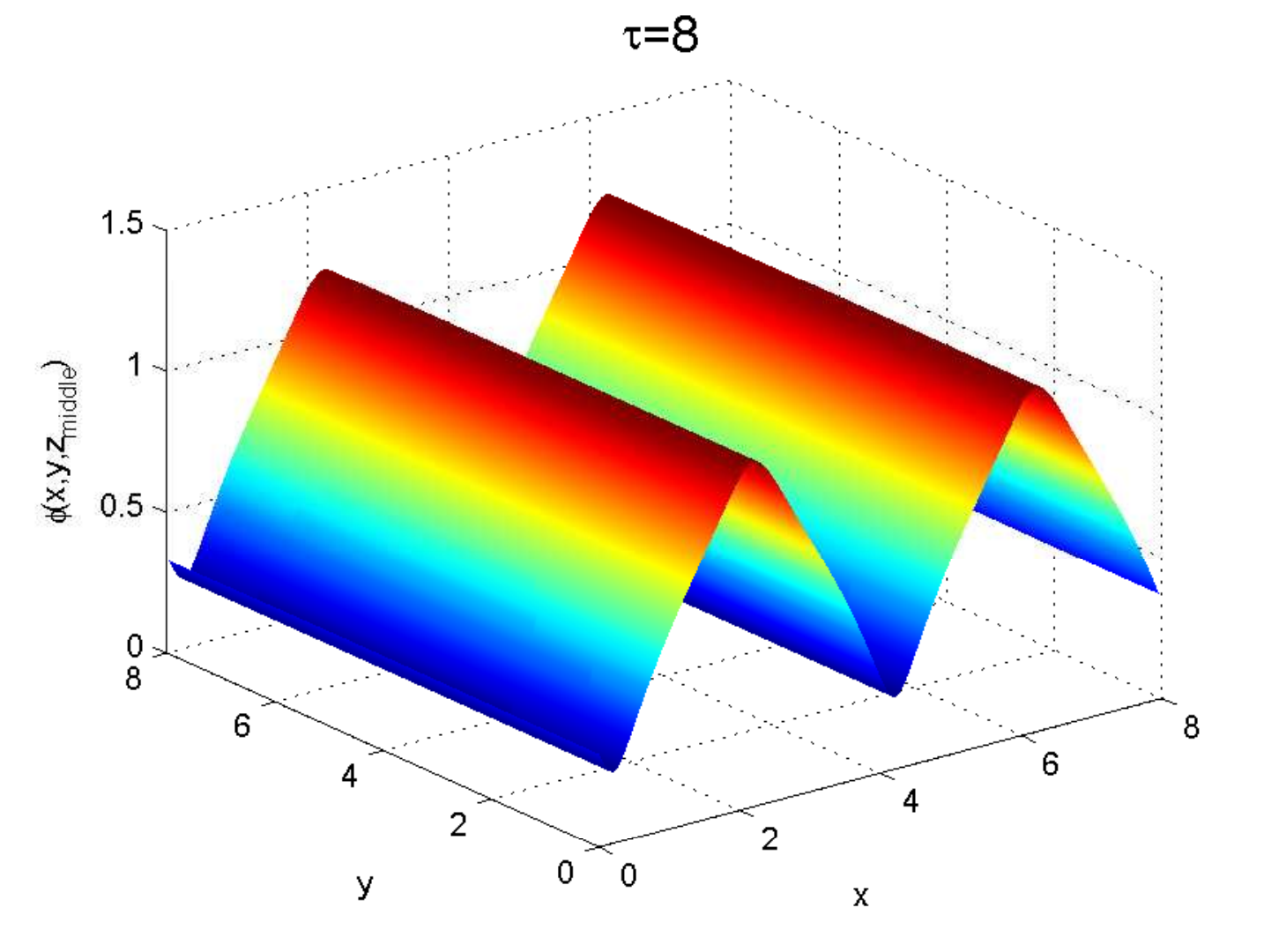,width=0.44\linewidth,clip=} \label{surface_h8}\\
        \epsfig{file=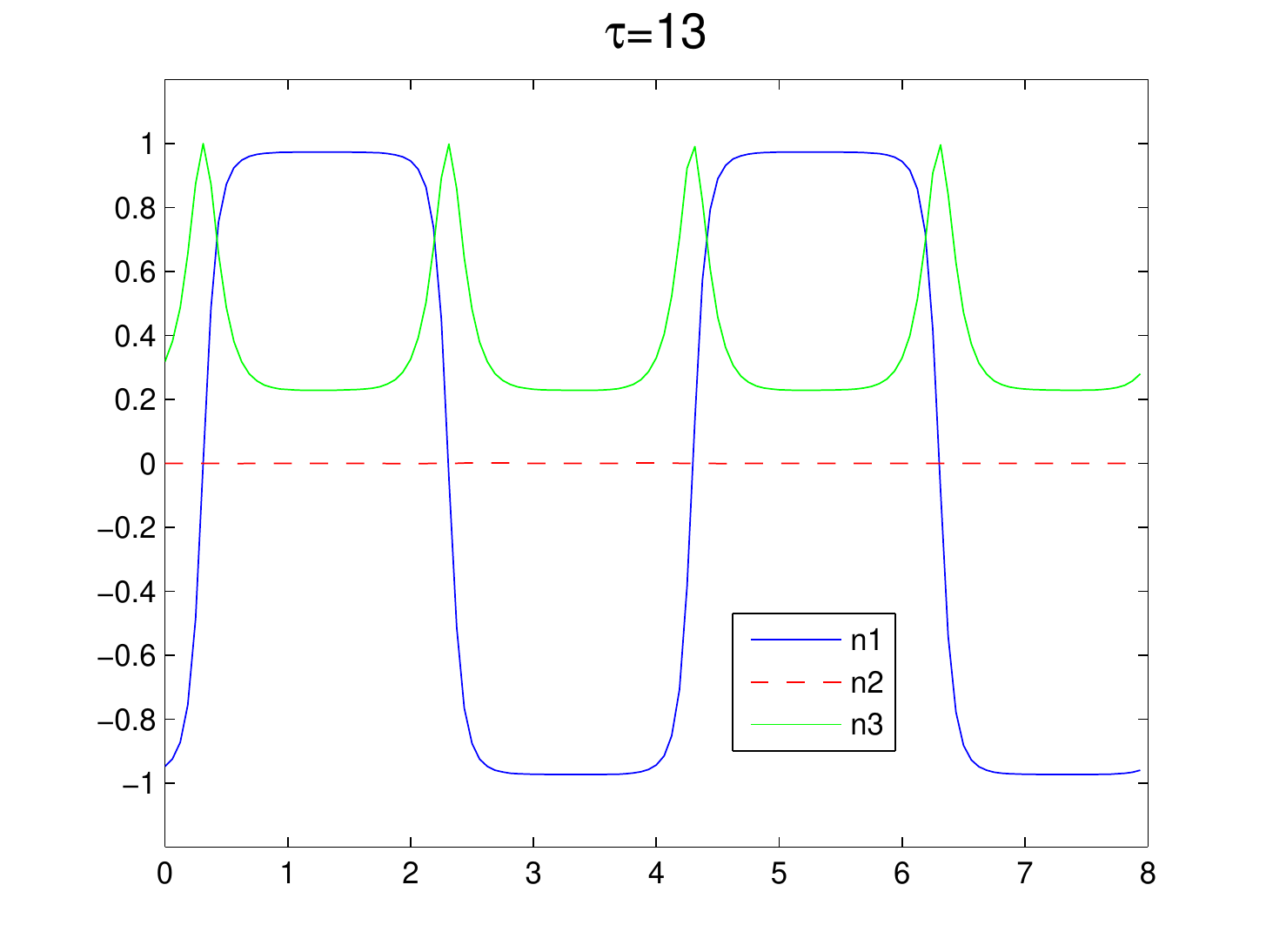,width=0.44\linewidth,clip=} \label{director_h13}&
        \epsfig{file=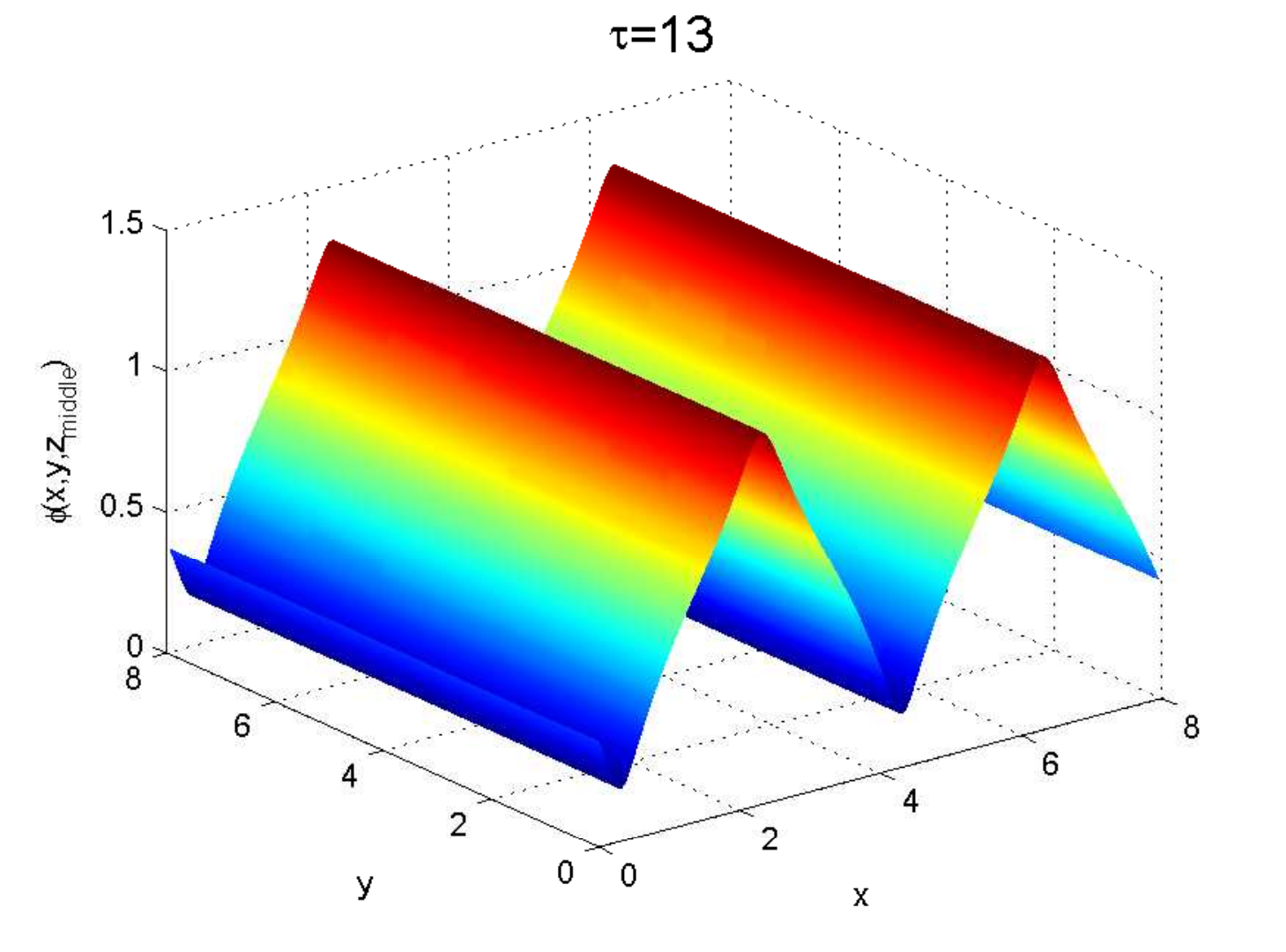,width=0.44\linewidth,clip=} \label{surface_h13}\\
        \epsfig{file=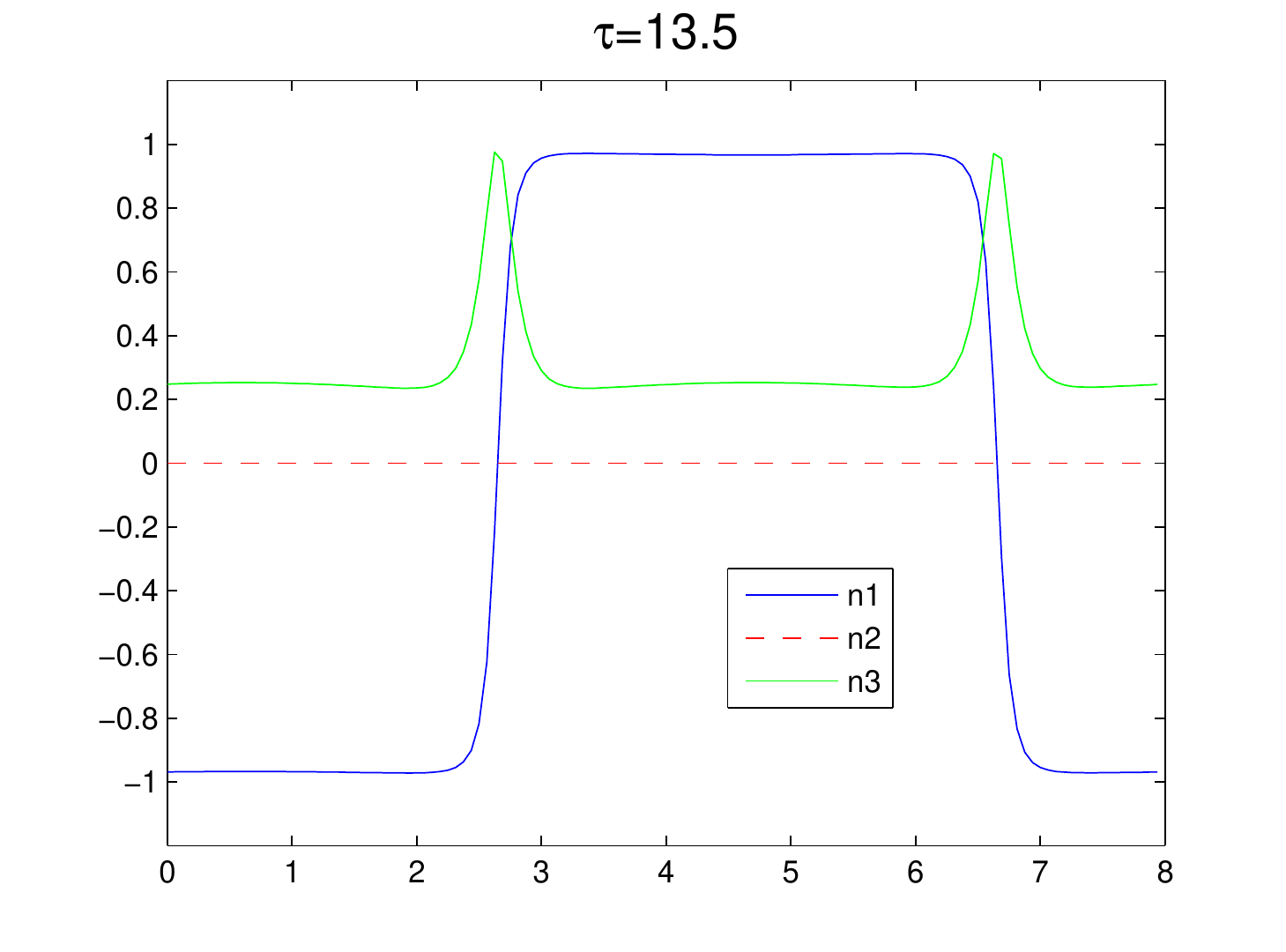,width=0.44\linewidth,clip=} \label{director_h13dot5}&
        \epsfig{file=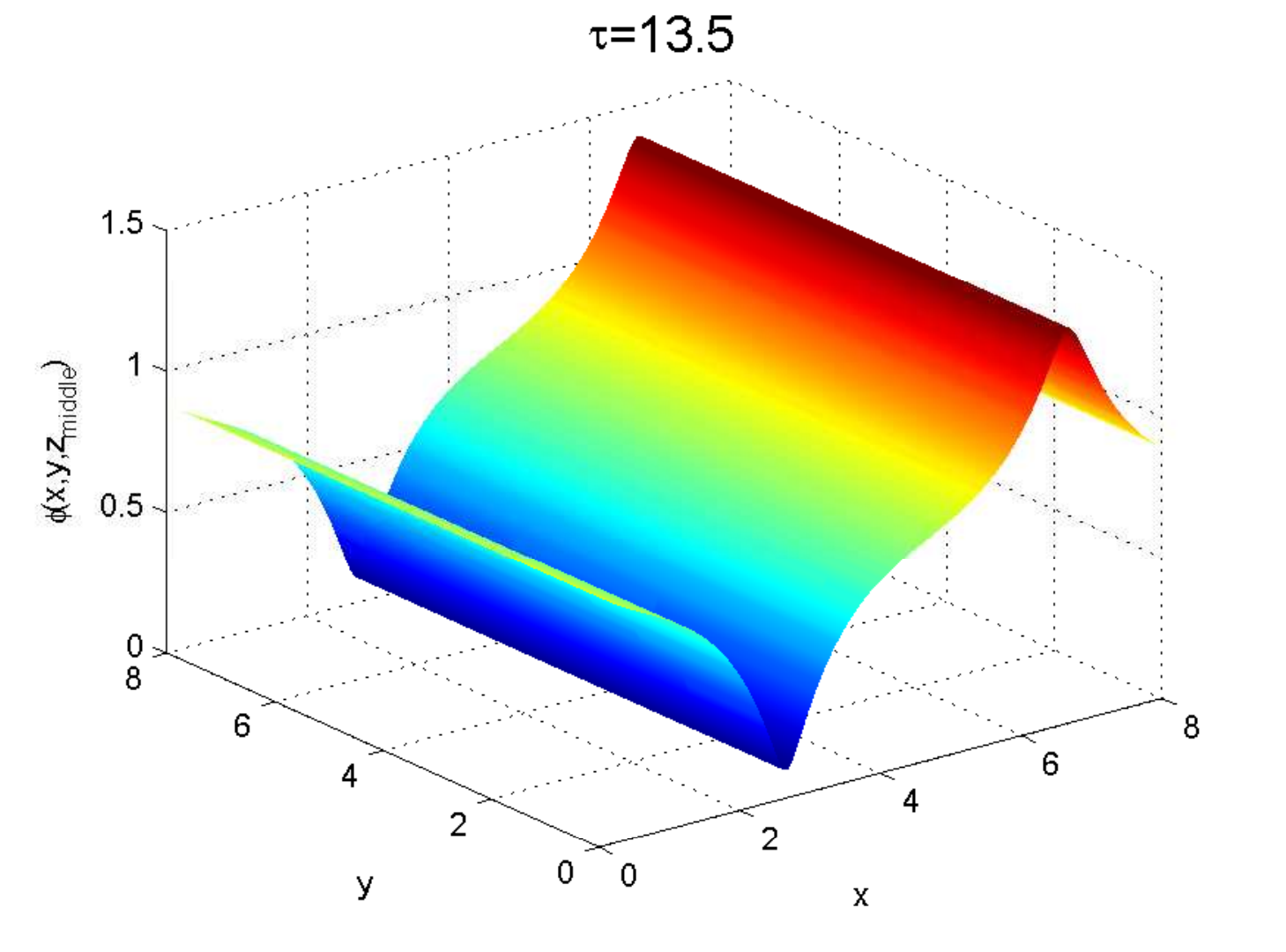,width=0.44\linewidth,clip=} \label{surface_h13dot5}\\
    \end{tabular}
    \caption{Chevron structures: Numerical solution of (\ref{GradientFlow}) with strong anchoring conditions on the bounding plates. The arrangement of rows are same as in Figure~\ref{fig:undulations3D}. The magnetic field strength $\tau = 8, 13, 13.5$ for each column. }
    \label{fig:chevron3D}
\end{figure}

Next, we look at the minimizers of the Chen-Lubensky free energy. The gradient flow equations associated with the energy (\ref{CL}) are given by
\begin{equation}
    \begin{aligned} \label{4th}
        \frac{\partial \bn}{\partial t}  =& -\bn \times \bn \times \Big( \e \Delta \bn - D_1 \e \nb ( \Delta \vp - \nb \cdot \bn) + \frac{D_2}{\e}|\nb \vp - \bn|^2 ( \nb \vp - \bn) \\
                 & \qquad \qquad \qquad \quad + \frac{1}{\e}(\nb \vp - \bn)+ \tau ( \bn \cdot \bh) \bh \Big),\\
        \frac{\partial \vp}{\partial t} =& -D_1 \e \Delta( \Delta \vp - \nb \cdot \bn) +  \frac{2 D_2}{\e} (\pa_j \vp - \bn_j) (\pa_{ij} \vp - \pa_i \bn_j)(\pa_i \vp - \bn_i) \\
        & \qquad \qquad + \frac{1}{\e}(D_2 |\nb \vp - \bn|^2 + 1) ( \Delta \vp - \nb \cdot \bn). \\
    \end{aligned}
\end{equation}
This system of fourth order partial differential equations has been studied in \cite{G-J3} to investigate the layer undulation phenomena. A new numerical  formulation was presented to reduce (\ref{4th}) to a system of second order equations with a constraint, which resembles the Navier-Stokes equations. The Gauge method (\cite{E-navier}) was adapted to solve the resulting equations. Details of the numerical methods can be found in \cite{G-J3}. We consider the rectangular domain $\omm = (-l, l) \times (-1,1)$ where $l = 4$. We employ Dirichlet boundary condition on $\bn = \mathbf{e}_3$ and $\vp = y$ at the $y= \pm 1$ and periodic boundary conditions at $x = \pm l$. The dimensionless parameters used are $D_1 = 0.1, \, D_2 = 0.76,$ and $ \e = 0.2$.

The first instability from the undeformed state ($\vp = y$, $\bn = \mathbf{e}_2$) is observed as layer undulations at $\tau = \pi$ as in the first row of Figure \ref{fig:chevronCL}. The stability analysis of the Chen-Lubensky free energy using $\Gamma$-convergence and bifurcation methods at the critical field is given in \cite{G-J3}.

\begin{figure}
\centering
    \begin{tabular}{cc}
        \epsfig{file=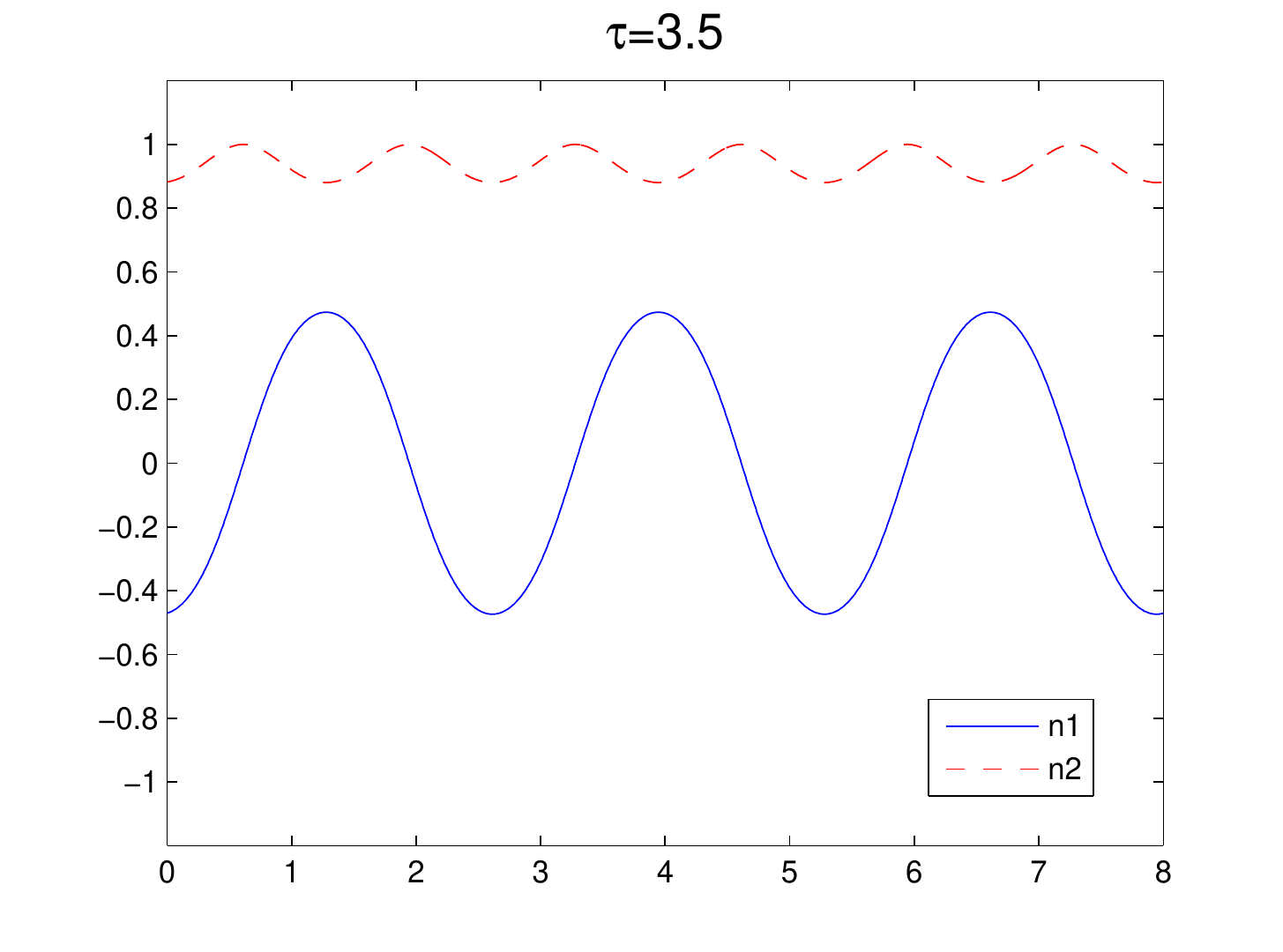,width=0.44\linewidth,clip=} &
       \epsfig{file=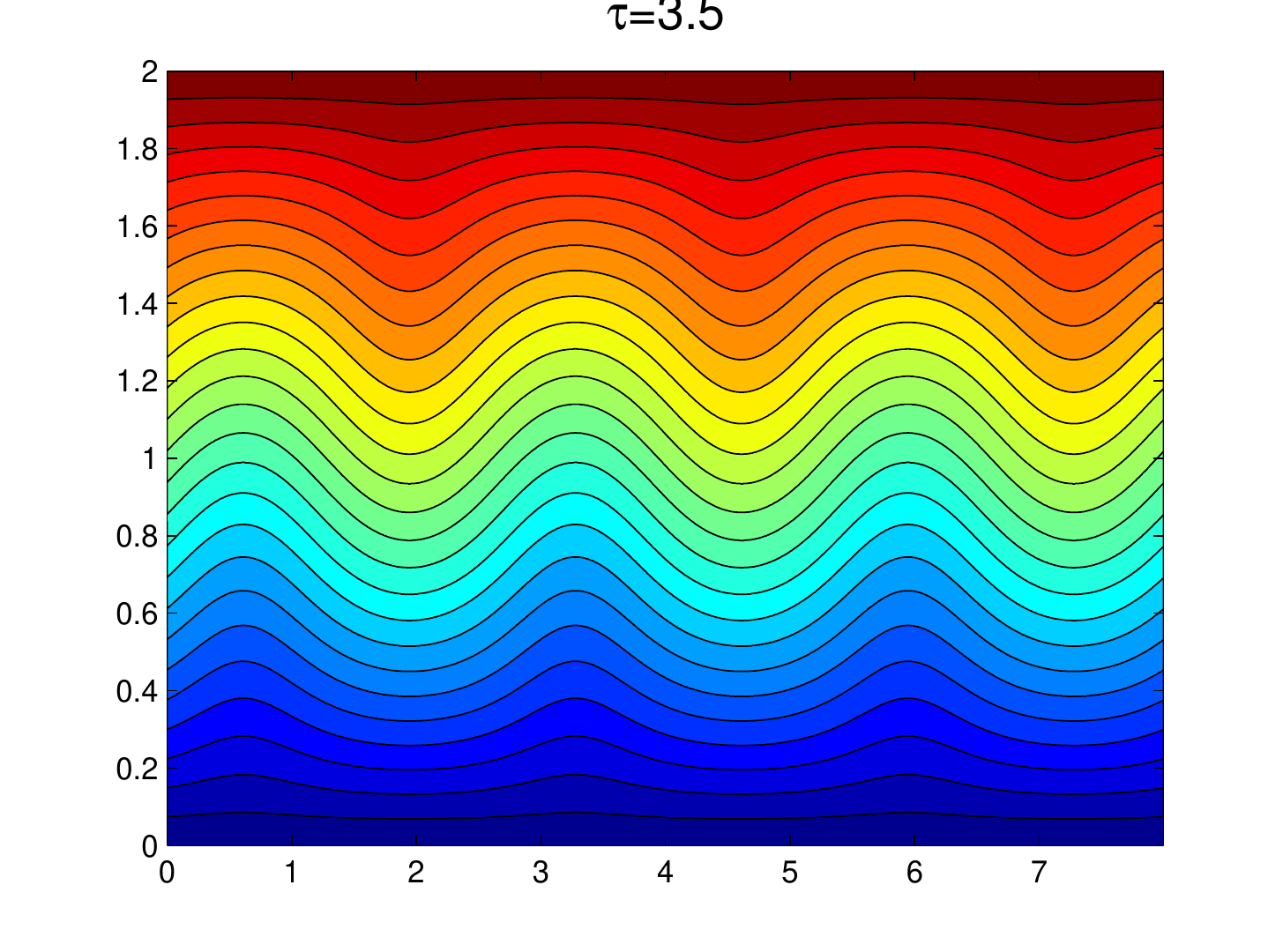,width=0.44\linewidth,clip=} \\
        \epsfig{file=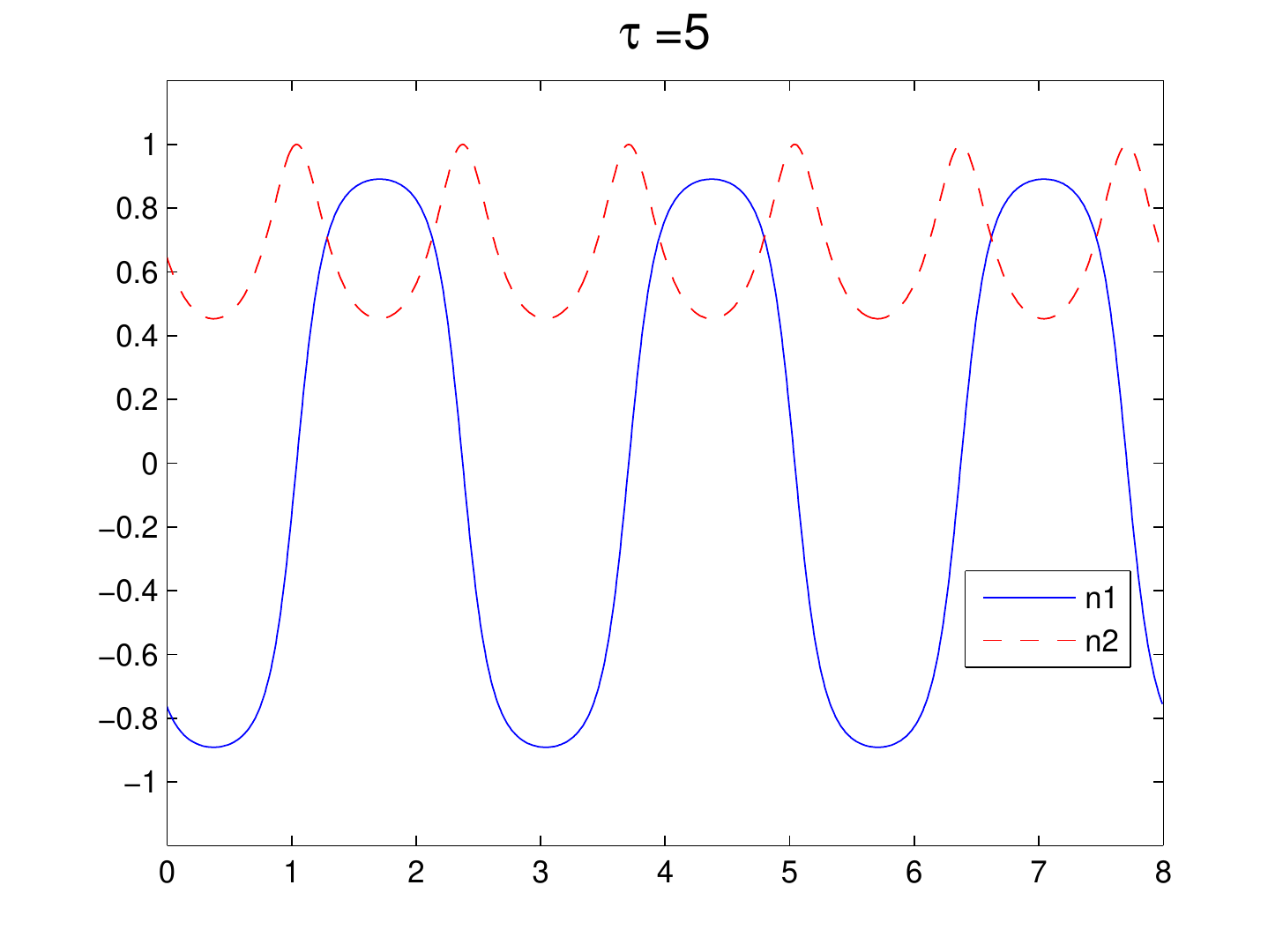,width=0.44\linewidth,clip=} &
        \epsfig{file=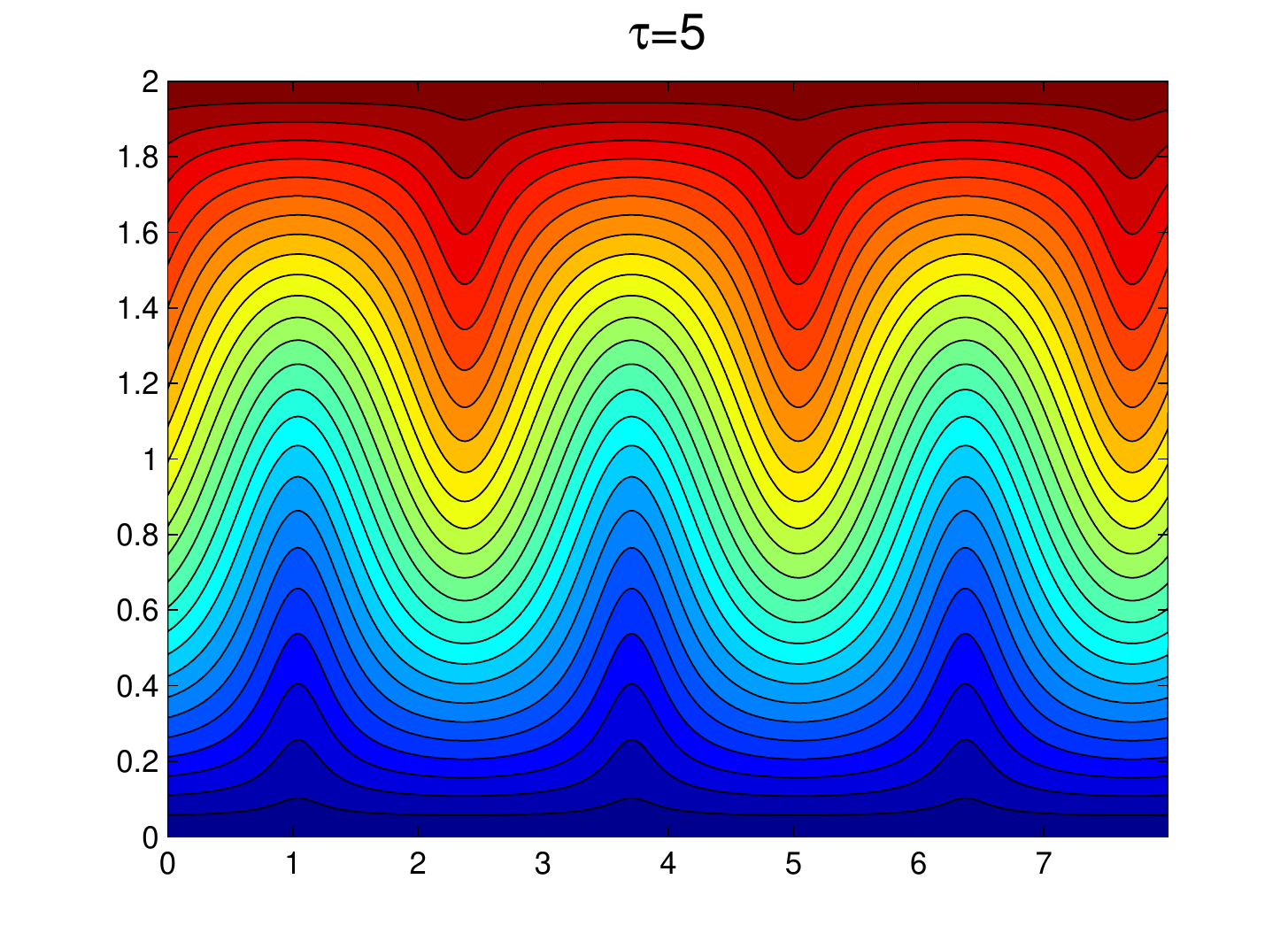,width=0.44\linewidth,clip=}\\
        \epsfig{file=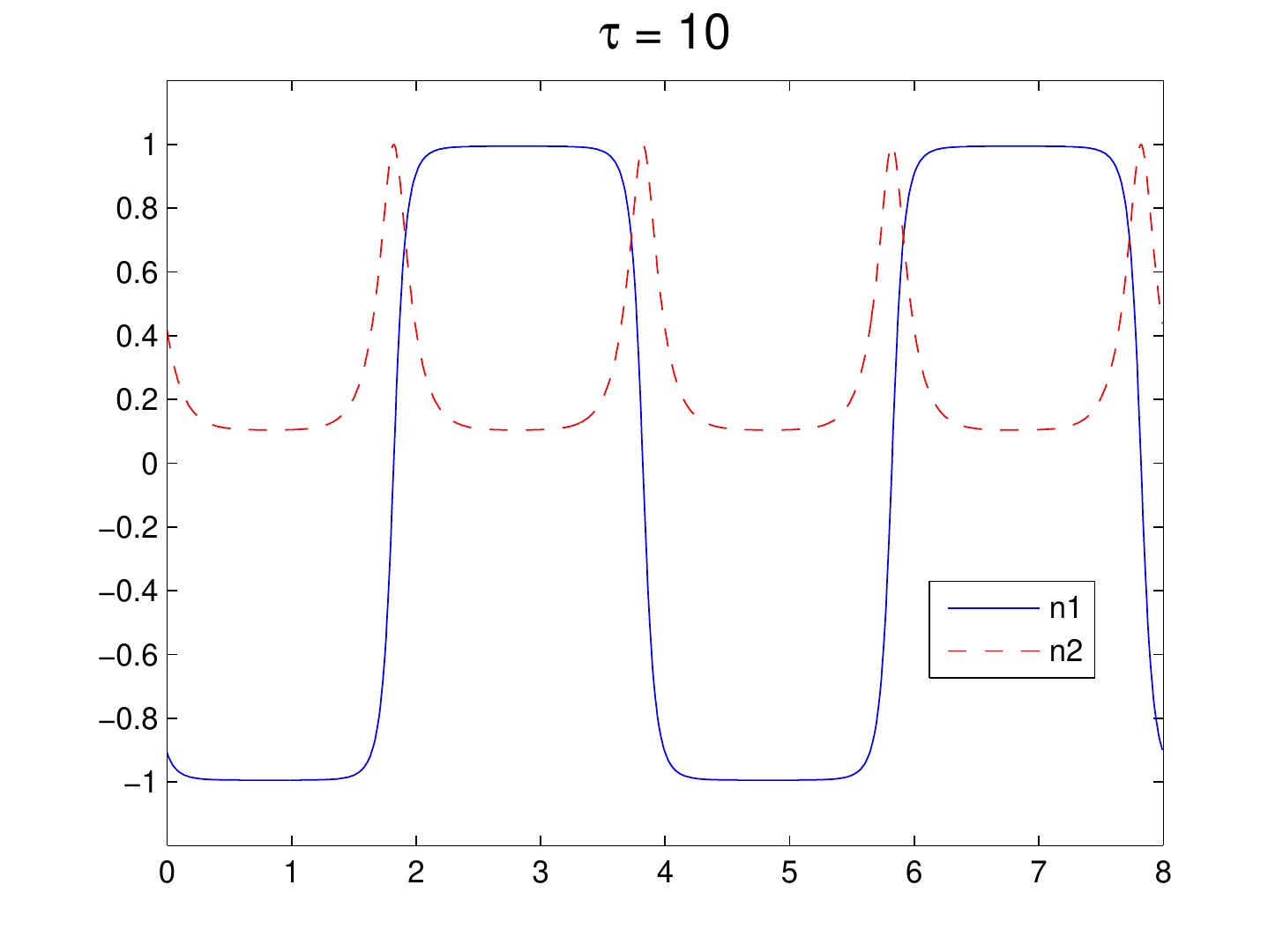,width=0.44\linewidth,clip=}&
        \epsfig{file=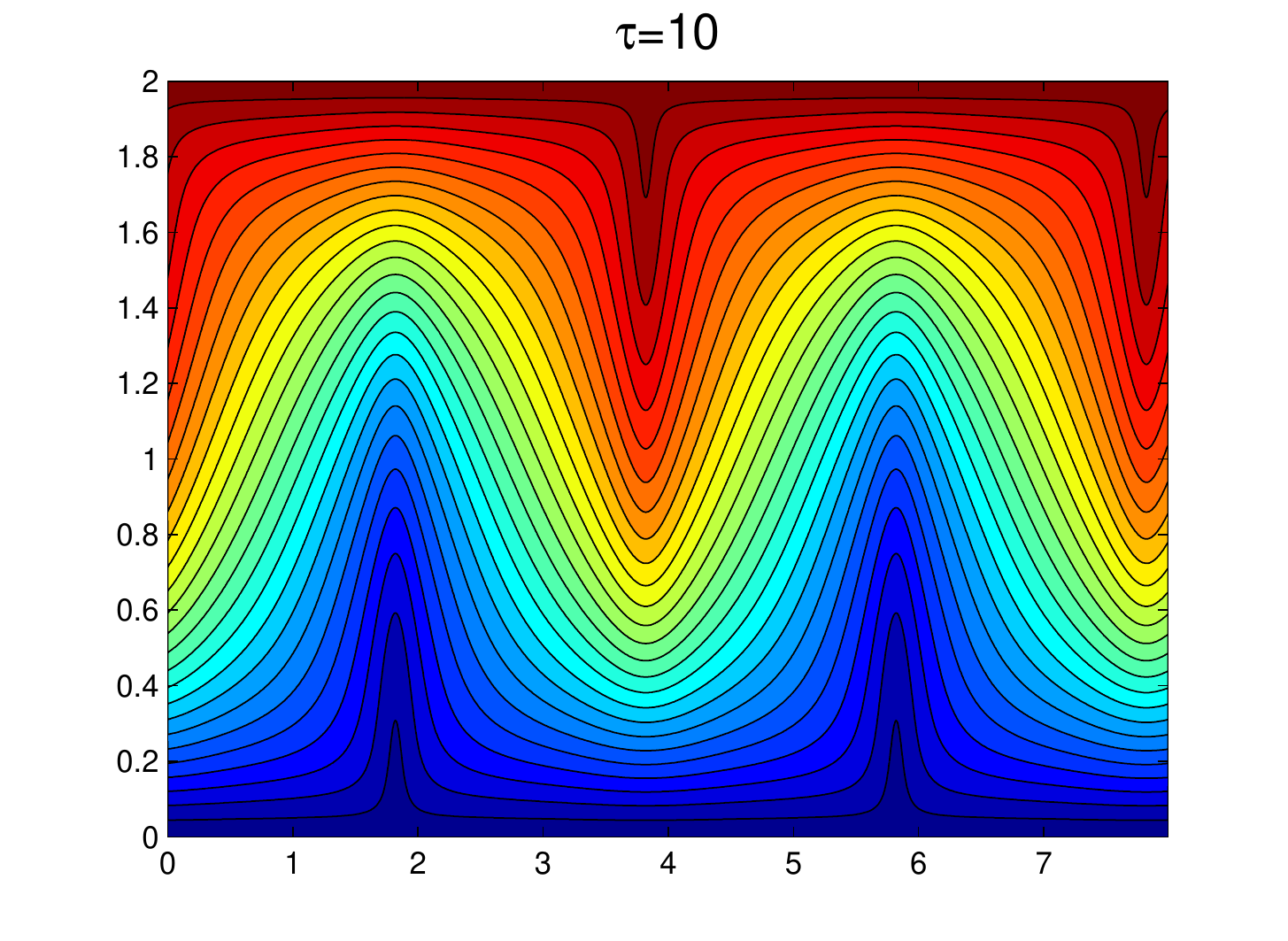,width=0.44\linewidth,clip=} \\
    \end{tabular}
    \caption{Numerical solution of (\ref{4th}) with strong anchoring conditions on the bounding plates. The first and second columns depict the scalar components of directors and layer in the middle of the cell. Onset of undulations in the first row, transformation from periodic oscillations to chevron structure in the second and third rows. }
    \label{fig:chevronCL}
\end{figure}

In the first column of Figure \ref{fig:chevronCL}, we depict the configuration of each component of $\bn$ in the middle of the domain, $y=0$. Also in here, one can clearly see that the undulatory pattern transforms to the zigzag pattern of the director. In the second column we show the layer description given by the contour map of $\vp$. In the middle of the domain the layer profile changes from sinusoidal to saw-tooth shape and its periodicity increases as the field strength increases.

In our analysis we let $\vp(x,y) = y-g(x)$ and find a minimizer of the energy for $g$ and $\bn$. Since the systems (\ref{GradientFlow}) and (\ref{4th}) are gradient flows of the energy in $\vp$ and $\bn$, we also present numerical simulations to find a minimizer of (\ref{CL_S1}) in a simpler setting, when $D_1 = D_2 =0$. We use a Truncated-Newton algorithm for energy minimization with a line search \cite{nocedal}. We use a Fourier spectral discretization in the $x$ direction. In Figure \ref{fig:chevron1D} we illustrate a chevron profile for $\theta$ and $g$, where $\bn = (\sin \theta, 0, \cos \theta)$.
\begin{figure}
\centering
    \begin{tabular}{cc}
        \epsfig{file=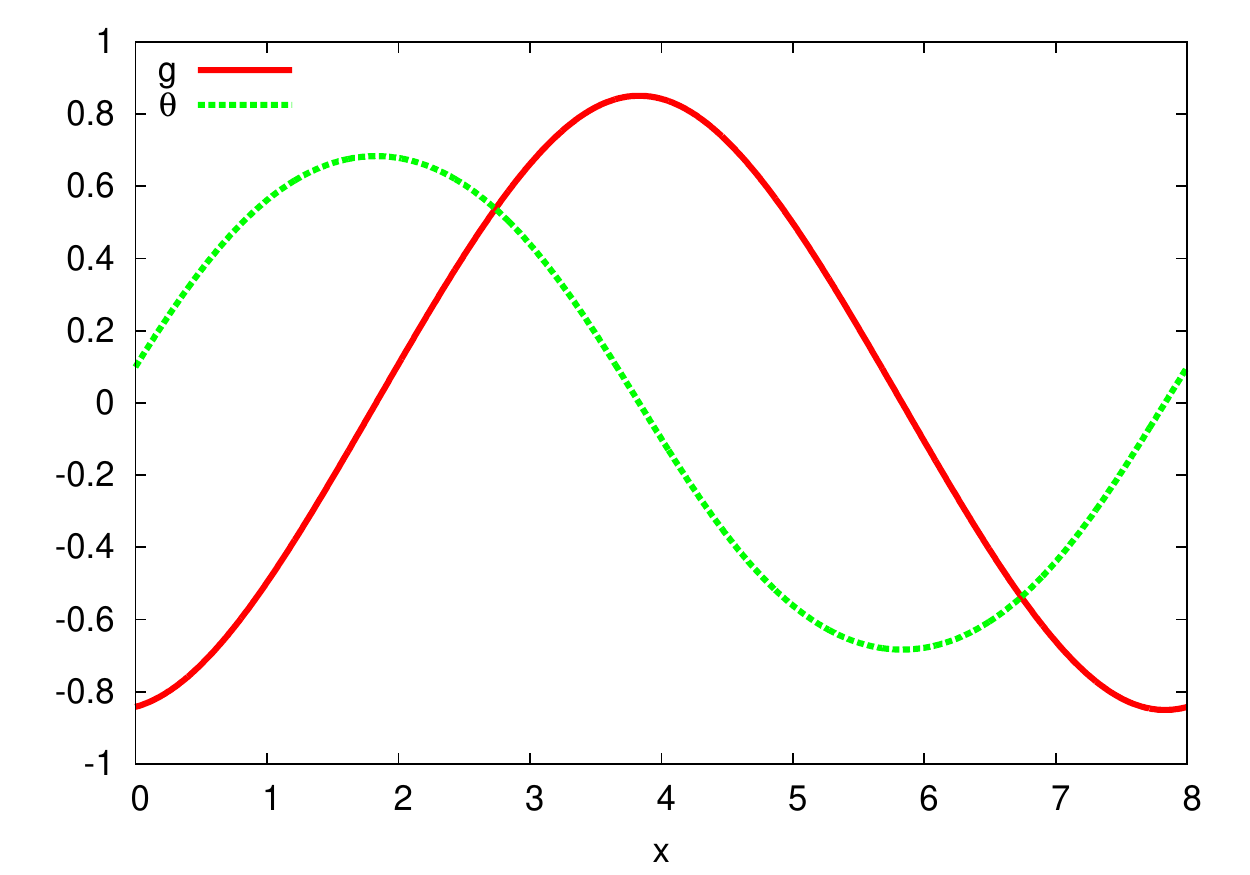,width=0.4\linewidth,clip=} &
        \epsfig{file=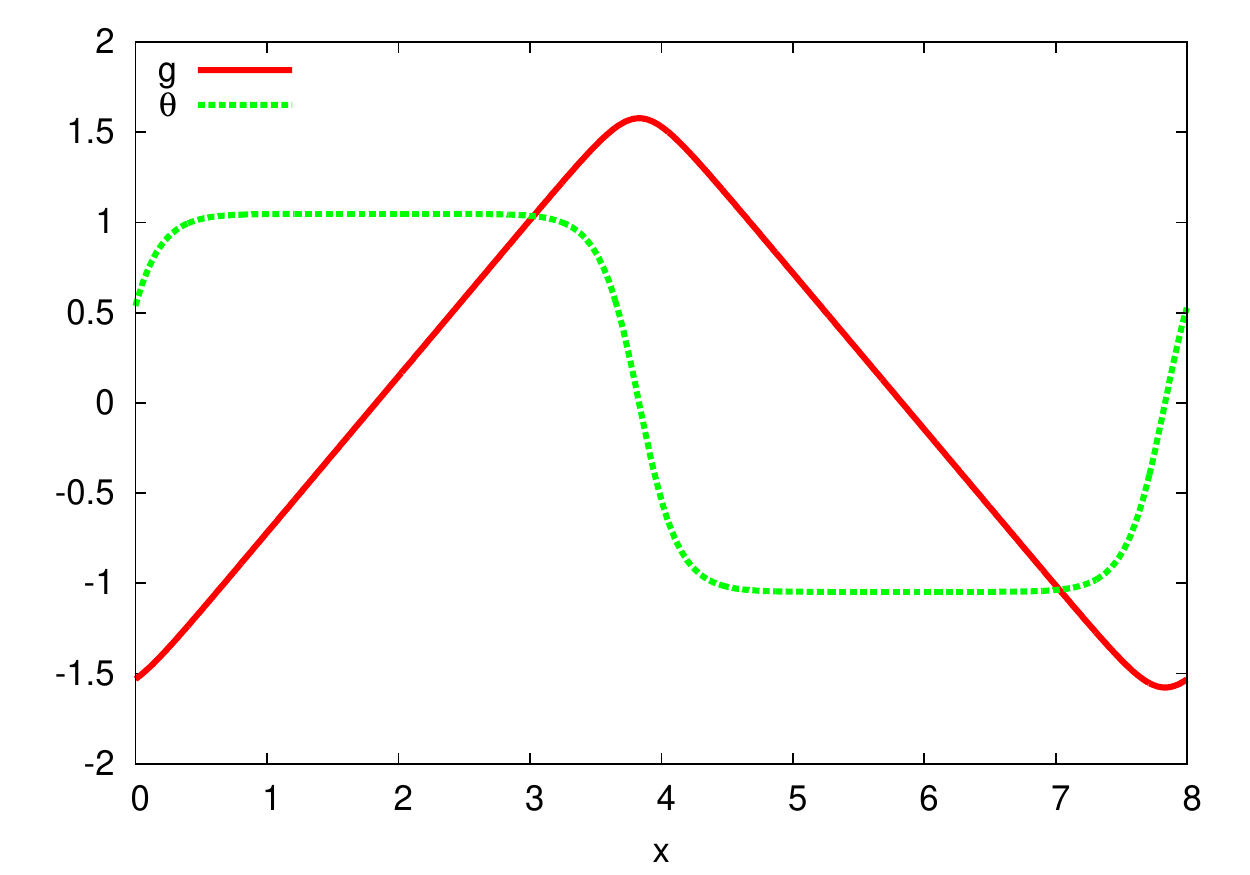,width=0.4\linewidth,clip=}\\{\bf(a)}& {\bf(b)}\\
    \end{tabular}
    \caption{Numerical minimizer of (\ref{CL_S1}) with $D_1 = D_2 = 0, \, A=0.5$ and (a) $\e = 1$ and (b) $\e = 0.2$. }
    \label{fig:chevron1D}
\end{figure}

%
%

\section{Appendix}
\subsection{Geodesics}\label{geodesics}
The construction of the recovering sequence in Theorem~\ref{thm:upper} in section~\ref{section:deGennes} is based on the explicit knowledge of a geodesic connecting the minima of the well potential of the de Gennes energy. Additionally, the value of the constant $c_0$ given in equation  (\ref{c_0}) is provided without a proof  in \cite{Moser2009},  we give below some ideas on how to derive these two facts.

Intuitively, one expects the infimum in $d({\bf n}_-, {\bf n}_+)$ to be achieved for the great arc
connecting ${\bf n}_-$ to ${\bf n}_+$, and  a direct computation gives (\ref{c_0}), by choosing the parametrization
\begin{equation} \label{geodesic-curve}
\gamma_C(t)=(\sin(2\alpha t- \alpha),0,\cos(2\alpha t-\alpha)).
\end{equation}
In fact,  the result is true for  the  distance
\begin{eqnarray}\label{weak-distance}
&& d_w(\xi,\eta)=\inf \bigg \{  \int_\gamma\, \sqrt{W} \, ds;  \mbox{ for } \gamma\in C^0([0,1]) \mbox{ and rectifiable, }  \nonumber \\
&&  \qquad \qquad \, \mbox{ such that } \gamma(t)\in S^2, \mbox{ and } \gamma(0)=\xi, \gamma(1)=\eta   \bigg \},
\end{eqnarray}
with
$
\displaystyle{ \int_\gamma\, \sqrt{W} \, ds = \lim \sum_{j=1}^n \sqrt{W(\gamma(r_j))} \, |\gamma(r_j)-\gamma(r_{j-1})|,}
$
where the limit is as in the definition of the Riemann integral (see \cite{Ahlfors} pages 104-105, for the two-dimensional case). But, from this one can conclude that the same holds for the distance $d$, since the parametrization $\gamma_C$ used in the direct computation mentioned above is $C^1$, any $\gamma\in C^1([0,1])$ is rectifiable, and for  any $\gamma \in C^1$ one has
$$
 \int_{\gamma}\, \sqrt{W} \, ds =  \int_0^1 \, \sqrt{W(\gamma(r))} \,\, |\gamma'(r)| \, dr.
$$

Fundamental in understanding why  the infimum is achieved along a great arc is to note how the function $W$ verifies the interesting property that its value at a generic point of the upper hemisphere having zero $y$-coordinate is smaller than the value at any other point on the sphere of same $x$-coordinate. 
\begin{lemma}\label{l5.1}
Given $P_0=(x_0,0,z_0)\in \mathbb{S}^2$ with $z_0>0$,  for any $P_0^*=(x_0,y_0^*,z_0^*)$  we have $W(P_0)\leq W(P^*_0)$.
\end{lemma}
\begin{proof}
By definition of $W$ and $\displaystyle{A=\frac 1{\sigma+1}}$, for a $P=(x,y,z) \in  \mathbb{S}^2$  we have
$$
W(P ) =   \sigma (1-x^2)+(z-1)^2-\frac \sigma{\sigma+1},
$$
and from this, one can see that for $x$ fixed, as $y$ decreases the constraint $P\in \mathbb{S}^2$ implies that $|z|$ increases, hence $W(P)$ decreases in $z$ for  $z>0$. On the other hand, for $P=(x,y,z)$ fixed with $z>0$, and $Q=(x,y,-z)$ clearly $W(P)<W(Q)$, so the lemma follows.
\end{proof}

 Let $\{\gamma_j\}$ be a minimizing sequence for $c_0=d_w({\bf n}_-, {\bf n}_+)$. We claim that the sequence can be chosen so that  the following conditions  are verified by any element $\gamma=(x(t),y(t),z(t)) \in  \{\gamma_j\}$:
\begin{itemize}
\item[(i)] $\gamma$ is such that $z(t) \geq 0$ for all $t$;
\item[(ii)] $\gamma$ is such that $y(t) \geq 0$ for all $t$,
\end{itemize}


Claim (i) follows from remarking that the $z$-coordinates of ${\bf n}_-$ and ${\bf n}_+$ are positive and the fact that for $P=(x,y,z)$ with $z>0$, and $Q=(x,y,-z)$  it holds $W(P)< W(Q)$, as seen  in the proof of Lemma~\ref{l5.1}, and so if we define
\begin{eqnarray*}
\hat{\gamma}(t)= \begin{cases}
	(x(t),y(t),z(t)) \qquad & \mbox{ if } z(t)\geq 0\\
      (x(t),y(t),-z(t)) & \mbox{ if } z(t)< 0,
           \end{cases}
\end{eqnarray*}
we have that $\hat{\gamma}$ is a $C^0([0,1])$ rectifiable  curve of shortest $d_w$-distance.


Claim (ii) is similarly true because the $y$-components of ${\bf n}_-$ and ${\bf n}_+$ are both zero, hence
\begin{eqnarray*}
\hat{\gamma}(t)= \begin{cases}
	(x(t),y(t),z(t)) \qquad & \mbox{if } y(t)\geq 0\\
        (x(t),-y(t),z(t)) \qquad & \mbox{if } y(t)<0,
           \end{cases}
\end{eqnarray*}
defines a  $C^0([0,1])$ rectifiable curve of shortest $d_w$-distance which verifies (ii).

\begin{lemma}\label{l5.2}
Let $\{\gamma_j\}$ denote a minimizing sequence for $d_w({\bf n}_-, {\bf n}_+)$, whose elements verify properties (i) and (ii) above, then for every $j$ it holds
\begin{eqnarray*}
 \int_{\gamma_j}\, \sqrt{W} \, ds \geq  \int_{\gamma_C}\, \sqrt{W} \, ds ,
\end{eqnarray*}
where $\gamma_C(t)=(\sin(2\alpha t- \alpha),0,\cos(2\alpha t-\alpha))$.
\end{lemma}
\begin{proof}
We denote by  $\gamma(t) =(x(t),y(t),z(t))$  a generic element of  $\{\gamma_j\}$,  and, for every $\epsilon>0$, pick $\eta_\epsilon^1>0$ such that for every partition
$P: r_0=0<r_1<...<r_{n-1}<r_n=1$,
with $\max_j|r_j-r_{j-1}|<\eta_\epsilon^1$, we have both
\begin{eqnarray*}
&& \Big|   \int_{\gamma}\, \sqrt{W} \, ds   - \sum_{j=1}^n \sqrt{W(\gamma(r_j))} \, |\gamma(r_j)-\gamma(r_{j-1})| \Big| <\epsilon,
\end{eqnarray*}
and
\begin{eqnarray*}
&& \Big| \int_{\gamma_C}\, \sqrt{W} \, ds   - \sum_{j=1}^n \sqrt{W(\gamma_C(r_j))} \, |\gamma_C(r_j)-\gamma_C(r_{j-1})| \Big| <\epsilon.
\end{eqnarray*}

Additionally,  $\gamma\in C^0([0,1])$ implies that there exists $\eta_\epsilon^2$ for which if $|r-p|<\eta_\epsilon^2$ then $\displaystyle{\Big|x(r)-x(p) \Big|< 2 {\alpha \, A}\, \eta^1_\epsilon}$, recall that $x(t)$ denotes the $x$-component of $\gamma$.

For $\epsilon>0$ fixed, pick $\eta_\epsilon < \min\{\eta^1_\epsilon, \eta^2_\epsilon\}$, and choose a partition $P_\epsilon: t_0=0<t_1<...<t_{n-1}<t_n=1$, with $\max_j |t_j-t_{j-1}|<\eta_\epsilon$. From this partition we build another partition $P^C_\epsilon:s_0<s_1<...<s_m$ as follows. We denote by $x_C$ the $x$-component of $\gamma_C$,
and set
$$
k_f=\inf\Big\{k:  t_k\in P_\epsilon \mbox{ and } x(t_k) \geq \sqrt{1-A^2} \Big\}\leq n,
$$
we then select  $s_0=k_0=0$, and for $j\geq 1$ we pick the first index $k_j$ such that $ x_C(s_{j-1})<x(t_{k_j})$. If $k_j=k_f$ we set $m=j$ and $s_j=s_m=1$, otherwise we continue and pick $s_j$ such that $x_C(s_j)=x(t_{k_j})$.
This process will stop after a finite number of steps $m\leq n$, returning $s_{m}=1$, as well as  $s_{j-1}<s_j$ for every $j$, since $x_C(s)$ is a strictly increasing function for $0<s<1$, and by construction $x_C(s_{j-1})<x_C(s_j)$.

The partition $P^C_\epsilon$ enjoys a few interesting properties, which we describe below. By construction
\begin{equation}\label{newproperty0}
 x(t_{{k_j}-1})\leq x_C(s_{j-1})< \sqrt{1-A^2},
 \end{equation}
 otherwise $t_{{k_j}-1}$ would have been picked at the $j$th step instead of $t_{k_j}$, as well as
 \begin{equation}\label{newproperty00}
 -\sqrt{1-A^2}=x_C(s_0) \leq  x(t_{{k_j}}).
 \end{equation}
And (\ref{newproperty0}), since $|t_{k_j}-t_{{k_j}-1}| < \eta^2_\epsilon$ by the choice of $\eta_\epsilon$, gives
 \begin{equation}\label{newproperty1}
0<x_C(s_j)-x_C(s_{j-1})\leq x(t_{k_j})-x(t_{{k_j}-1})< 2 {\alpha \, A}\,  \,\eta^1_\epsilon.
\end{equation}
In turn, equation (\ref{newproperty1}), with $-\sqrt{1-A^2}\leq x_C(s_j) \leq \sqrt{1-A^2}$ and $x_C(s_j)=\sin(2\alpha \, s_j -\alpha)$ implies
\begin{eqnarray}
&&|s_j-s_{j-1}|=\frac1{2\alpha} \Big|\arcsin(x_C(s_j))-\arcsin(x_C(s_{j-1}))\Big| \nonumber \\
\nonumber\\
&& \qquad \leq \frac 1{2\alpha \, A}\Big|x_C(s_j)-x_C(s_{j-1}) \Big|< \,\eta^1_\epsilon,
\end{eqnarray}
as $\Big|\arcsin(x_C(s_j))-\arcsin(x_C(s_{j-1}))\Big| \leq \frac{1}{\sqrt{1-\xi^2}}\Big|x_C(s_j)-x_C(s_{j-1})\Big|$ for some
$x_C(s_{j-1})\leq \xi \leq x_C(s_{j})$.

Finally,   it is possible to show that
\begin{equation}\label{newproperty2}
|\gamma_C(s_j)-\gamma_C(s_{j-1})|\leq |\gamma(t_{k_j})-\gamma(t_{k_j-1})|.
\end{equation}
To see this, we first notice that  since $0<\alpha<\frac\pi 2$, by picking  $\eta^1_\epsilon$ small enough,  because of (\ref{newproperty0}) and (\ref{newproperty00}), we can assume
$$
-1<- \sqrt{1-A^2} -2\alpha A  \eta^1_\epsilon <  x(t_{{k_j}-1}) < \sqrt{1-A^2},
$$
so that there exists a largest  $s_-$ such that $s_- \leq s_{j-1}$ and $x_C(s_-)= x(t_{{k_j}-1})$, note here $s_-$ could be smaller than 0 but not larger than 1.
But even if $s_-<0$, again by taking  $\eta^1_\epsilon$ sufficiently small, by the continuity of $x_C$, we can assume $0\leq 2 \, \alpha \, (s_j-s_{j-1})\leq 2 \, \alpha \, (s_j -s_-) \leq \pi $, so that
$$
\cos(2 \, \alpha \, (s_j-s_{j-1})) \geq \cos(2 \, \alpha \, (s_j -s_-) ),
$$
from which, since a direct computation for every $p, r$ gives
$$
|\gamma_C(p)-\gamma_C(r)|^2=2-2\cos(2\, \alpha \, (p-r))
$$
we obtain
\begin{equation}\label{newproperty2-step1}
|\gamma_C(s_j)-\gamma_C(s_{j-1})|\leq |\gamma_C(s_j)-\gamma_C(s_{-})|.
\end{equation}
Therefore, to derive (\ref{newproperty2}) it will be enough to prove
\begin{equation}\label{newproperty2-step2}
 |\gamma_C(s_j)-\gamma_C(s_{-})|\leq |\gamma(t_{k_j})-\gamma(t_{k_j-1})|
\end{equation}

We set $\gamma_C(s_j)=(\phi,0,\psi)$, $\gamma_C(s_-)=(\alpha_0,0,\eta)$, $\gamma(t_{k_j})=(\phi,\chi_*,\psi_*)$, and $\gamma(t_{k_j-1})=(\alpha_0,\beta_*,\eta_*)$,  and remind  the reader that these are points on $\mathbb{S}^2$, for which $\chi_*,\psi_*,\beta_*$,$\eta_*$ and $\alpha_0$ are positive, since $\gamma$ satisfies conditions (i) and (ii). Therefore, we can rewrite
\begin{eqnarray*}
&& (\psi-\eta)^2=\psi^2-2 \psi \eta+\eta^2=\chi_*^2+\psi_*^2-2 \psi \eta +\beta_*^2+\eta_*^2 \\
&& \quad = (\chi_*-\beta_*)^2+(\psi_*-\eta_*)^2 -2 \psi \eta +2 \chi_* \beta_* + 2 \psi_* \eta_*,
\end{eqnarray*}
so that
\begin{eqnarray*}
&&|\gamma_C(s_j)-\gamma_C(s_-)|^2= |\gamma(t_{k_j})-\gamma(t_{k_j-1})|^2-2 \psi \eta +2 \chi_* \beta_* + 2 \psi_* \eta_*,
\end{eqnarray*}
and inequality (\ref{newproperty2}) follows by noticing that
\begin{eqnarray*}
&& \psi^2 \, \eta^2  = (\chi_*^2+\psi_*^2) (\beta_*^2+\eta_*^2)\, =\, \chi_*^2 \beta_*^2 +\chi_*^2 \eta_*^2+\psi_*^2\beta_*^2 +\psi_*^2 \eta_*^2\\
&& \qquad = \, \chi_*^2 \beta_*^2 +\psi_*^2 \eta_*^2+ (\chi_* \eta_* - \psi_*\beta_*)^2 + 2 \chi_* \eta_*\psi_*\beta_*\\
&& \qquad =(\chi_*\beta_*+\psi_*\eta_*)^2+(\chi_* \eta_* - \psi_*\beta_*)^2 \geq (\chi_*\beta_*+\psi_*\eta_*)^2.
\end{eqnarray*}
In conclusion,  for every $\epsilon>0$, considering the partitions $P_\epsilon$ and $P^C_\epsilon$ and using Lemma~\ref{l5.1}, we have
\begin{eqnarray*}
&& \int_{0}^{1} \, \sqrt{W(\gamma(t))} \,\, |\gamma'(t)| \, dt  \geq  \sum_{j=1}^n \sqrt{W(\gamma(t_j))} \, |\gamma(t_j)-\gamma(t_{j-1})| - \epsilon \\
&&\quad \geq \sum_{j=1}^m \sqrt{W(\gamma(t_{k_j}))} \, |\gamma(t_{k_j})-\gamma(t_{k_j-1})| - \epsilon \\
&& \quad \geq  \sum_{j=1}^m \sqrt{W(\gamma_C(s_{j}))} \, |\gamma(t_{k_j})-\gamma(t_{k_j-1})| - \epsilon \\
&&\quad  \geq \sum_{j=1}^m \sqrt{W(\gamma_C(s_{j}))} \, |\gamma_C(s_{j})-\gamma_C(s_{j-1})|- \epsilon  \\
&& \quad \geq \int_{0}^{1} \, \sqrt{W(\gamma_C(t))} \,\, |\gamma_C'(t)| \, dt - 2 \epsilon,
\end{eqnarray*}
from which the lemma follows.
\end{proof}

\subsection{Chen-Lubensky double well potential}\label{CL:well}

To analyze the zeros of the function $W$ in (\ref{well-CL}),
$$
W(\theta) = 8D_2\sin^8\frac{\theta}{2} + 4(1+\sigma) \sin^4 \frac{\theta}{2} - 4\sigma\sin^2 \frac{\theta}{2} +a_0,
$$
we set $x=\sin{\theta}/{2}$, and look at the critical points of
$$
f(x) = 8\, D_2 x^8+ 4(1+\sigma) x^4 - 4\sigma x^2.
$$
Taking the derivative, we find
$$
f'(x) = 64 \, D_2 \, x\, \Big(x^6 + \frac 1{4 \, D_2} (\sigma+1) \, x^2 - \frac\sigma{8\, D_2}\Big),
$$
from which, setting $u=x^2$ we are led to the cubic polynomial equation
\begin{equation}\label{depress}
u^3+ \frac 1{4 \, D_2} (\sigma+1) \, u - \frac{\sigma}{8\, D_2}=0.
\end{equation}
Since $\frac 1{4 \, D_2} (\sigma+1)>0$, this equation has only one real root given by the formula
\begin{equation} \label{expression-r}
R(\sigma,D_2)=\sqrt{\frac{\sigma+1}{3 \, D_2}} \, \sinh \Bigg(\frac 13 \arsinh \Bigg(\frac 32\, \frac \sigma{\sigma+1}\,\sqrt{\frac{3 \, D_2}{\sigma+1}}\Bigg) \Bigg).
\end{equation}
Using direct computations, a contradiction argument, and the limit
$$
\lim_{t\to 0} \frac{\sinh \Big(\frac 13 \arsinh \Big(\frac 32\,t\Big) \Big)}{t}  = \frac 12,
$$
it's straightforward to see that
\begin{eqnarray*}
&& R(0,D_2)=0, \, \lim_{\sigma \to \infty} R(\sigma,D_2)=\frac 12,  \\
&&  \lim_{D_2 \to 0^+} R(\sigma,D_2)=\frac{\sigma}{2(\sigma+1)}, \,  0\leq R(\sigma,D_2)< \frac 12,
\end{eqnarray*}
and $f$ is a symmetric even function with a maximum at $x=0$ and minima at $x=\pm \sqrt{R}$.

In terms of $W(\theta)$, picking $a_0=-f(\sqrt{R})$, we have that  for fixed $\sigma$ and $D_2$ positive, $W$ is an symmetric even function of $\theta$, and has only two zeros of opposite sign:
\begin{equation} \label{expression-beta}
\pm \beta(\sigma,D_2)=\pm 2 \arcsin\sqrt{R(\sigma,D_2)},
\end{equation}
which verify
\begin{eqnarray}
&& \beta(0,D_2)=0, \quad \lim_{\sigma \to \infty} \beta(\sigma,D_2)=\frac \pi 2,\nonumber \\
&& \lim_{D_2 \to 0^+} \beta(\sigma,D_2)=\alpha, \quad  0\leq \beta(\sigma,D_2)< \frac \pi 2.
\end{eqnarray}

\bibliographystyle{plain}
\bibliography{chevron_lib}

\end{document}